\documentclass[oneside,11pt]{article}
\usepackage{epsfig, amsfonts, amsmath, amsthm,amsbsy}
\usepackage{color}
\usepackage{subfigure}
\newtheorem{theorem}{Theorem}

\newtheorem{remark}{Remark}
\newtheorem{example}{Example}
\def\ve{\varepsilon}
\def\vr{\varepsilon}
\def\ds{\displaystyle}
 \DeclareMathOperator\erfc{erfc}

\begin{document}

 \title{\bf Numerical approximations to a singularly perturbed convection-diffusion problem with a discontinuous initial condition\thanks{This research was partially supported  by the Institute of Mathematics and Applications (IUMA), the projects PID2019-105979GB-I00 and PGC2018-094341-B-I00 and the Diputaci\'on General de Arag\'on (E24-17R).}}
\author{J.L. Gracia\thanks{Department of Applied Mathematics, University of
Zaragoza, Spain.\  email: jlgracia@unizar.es} \and E.\ O'Riordan\thanks{School of Mathematical Sciences, Dublin City
University, Dublin 9, Ireland.\ email: eugene.oriordan@dcu.ie}}

\maketitle

\begin{abstract}

A singularly perturbed parabolic problem of convection-diffusion type with a discontinuous initial condition is examined. An analytic function is identified which matches the discontinuity in the initial condition and also satisfies the homogenous parabolic differential equation associated with the problem. The difference between this analytical function and the solution of the parabolic problem is approximated numerically, using an upwind finite difference operator combined with
an appropriate layer-adapted mesh.   The numerical method is shown to be parameter-uniform. Numerical results are presented to illustrate the theoretical error bounds established in the paper.

\


\end{abstract}

\section{Introduction}

In this paper, we examine a singularly perturbed convection-diffusion problem with a discontinuous initial condition. Throughout, we assume that the convective coefficient multiplying the first spatial derivative of the solution (denoted below by $u(x,t)$) is smooth, strictly positive and depends solely on the time variable $t$. Under this assumption, an explicit discontinuous function (denoted below by $S(x,t)$) can be identified which captures the nature of the singularity associated with the discontinuous initial condition. Asymptotic expansions for the solution $u(x,t)$, involving this singularity, were constructed in \cite{bobisud}.
 When we subtract off this singular function the remainder $y(x,t):=u(x,t) -S(x,t)$ is again the solution of  a singularly perturbed convection-diffusion problem. However, although the remainder $y$  satisfies the same singularly perturbed partial differential equation as $u$,  the  initial condition $y(x,0)$ is  now continuous. In this paper, we construct and analyze a numerical method that produces parameter-uniform \cite{fhmos} numerical approximations to this remainder $y$.

 In  \cite{jcam 2020}, we examined a set of related singularly perturbed reaction-diffusion problems with  discontinuities in either the boundary or the initial condition. In this paper, we extend this technique to a convection-diffusion problem with a discontinuous initial condition. In this case, the location of the interior layer (generated by the discontinuity in the initial condition)   moves in time and, in addition, this interior layer can eventually merge into a boundary layer.

Shishkin \cite{shishkin4} constructed and analysed a numerical method for the problem examined below, in the case where the initial condition is continuous, but has a discontinuity in the first derivative. The error bound in \cite[(5.23)]{shishkin4} essentially coincides with the error bound presented below in Theorem 1. Hence,
the presence of the complimentary error function in the problem formulated here does not significantly alter the final error bound established in \cite[Theorem 8]{shishkin4}.

In the more general case where the convective coefficient can depend on both space and time, remainder $y$ contains a strong interior layer
and the numerical algorithm presented in this paper will not suffice (see Example~\ref{ex5} in \S \ref{sec:numer-extension})  to generate parameter-uniform approximations. In a companion paper \cite{arxiv} to the current paper, we present a different numerical algorithm to manage this more general case.  Here we show that a simpler algorithm (to the algorithm in \cite{arxiv}) suffices when the convective coefficient does not depend on the spatial variable $x$.

In   \S2, we define the continuous problem to be examined, define the singular function $S(x,t)$  and present a priori bounds on the derivatives of the remainder term $y(x,t)$. In \S 3, we construct a numerical method and establish a parameter-uniform  error bound on the associated numerical approximations.  In \S 4,  we present some numerical results to illustrate the performance of the numerical method and to support the theoretical error bounds.  Some technical details are available in the Appendix.

{\bf Notation:} Throughout the paper, $C$ denotes a generic constant that is independent of the singular perturbation parameter $\vr$ and all the discretization parameters. The $L_\infty$ norm on the domain $D$ will be denoted by $\Vert \cdot \Vert_D$. We also define the jump of a function at a point $d$ by  $ [\phi ] (d) := \phi (d^+) - \phi (d^-) $.

\section{Continuous problem}

Consider the following singularly perturbed parabolic convection-diffusion problem:
 Find $ u$  such that \footnote{ As in \cite{friedman}, we define the space ${\mathcal C}^{0+\gamma}(D )$, where $D \subset \mathbf{R}^2$ is an open set, as the set of all functions that are H\"{o}lder continuous of degree $\gamma \in (0,1) $ with respect to the metric $\Vert \cdot \Vert, $
where for all ${\bf p}_i=(x_i,t_i),  \in \mathbf{R}^2, i=1,2; \
\Vert {\bf p}_1- {\bf p}_2 \Vert^2  = (x_1-x_2)^2 + \vert t_1 -t_2 \vert$.
For $f$ to be in ${\mathcal C}^{0+\gamma}(D ) $ the following semi-norm needs to be finite
\[
\lceil f \rceil _{0+\gamma , D} = \sup _{{\bf p}_1
 \neq {\bf p}_2, \ {\bf p}_1, {\bf p}_2\in D}
\frac{\vert f({\bf p}_1) - f({\bf p}_2) \vert}{\Vert {\bf p}_1- {\bf p}_2 \Vert ^\gamma} .
\] The space ${\mathcal C}^{n+ \gamma}(D ) $ is defined by
\[
{\mathcal C}^{n+\gamma }( D ) = \left \{ z : \frac{\partial ^{i+j} z}{
\partial x^i
\partial t^j } \in {\mathcal C}^{0+\gamma }(D), \ 0 \leq i+2j \leq n \right \},
\]
and $\Vert \cdot \Vert _{n + \gamma}, \ \lceil \cdot \rceil _{n+\gamma}$ are the associated norms and semi-norms.
}
\begin{subequations}\label{problem2}
\begin{align}
 & L  u := -\ve {u}_{xx} +  a(t) {u}_x +{u}_t= 0, \quad (x,t) \in  Q:=(0,1)\times (0,T],\label{1a} \\
&{u}(x,0)=\phi (x), \, 0 \leq x \leq 1;  \   [\phi ](d)\neq 0,\ 0 < d =O(1) <1; \label{1b} \\
&{u}(0,t)=0= {u}(1,t),  \ 0 < t \leq T,\label{1c}\\
& a(t) > \alpha >0, \ 0 \leq t \leq T, \quad  a \in C^{4+\gamma } (\bar{ Q} ),   \label{1d}\\
&\ \ \phi \in C^4((0,1) \setminus \{d \});\qquad \phi ^{(2i)} (0)=\phi ^{(2i)}(1) =0;\quad i=0,1,2. \label{comp0}
\end{align}
\end{subequations}
The assumption of the compatibility conditions (\ref{comp0}) ensures that  no classical singularity appears near the  end points
 $(0,0), (1,0)$. Observe that the initial function $\phi(x)$ is discontinuous at $x=d$. This will cause an interior layer to appear in the solution, near the point $(d,0)$, which will be convected into the interior of the domain
along the characteristic curve
\[
\Gamma ^* := \{ (d(t),t) | \, d'(t) =  a(t), \ 0<d(0)=d<1 \}
\]
associated with the reduced first order differential equation.
Define the continuous function
\begin{equation}\label{def-S}
 y (x,t) :=  u(x,t) - S(x,t), \qquad S(x,t):= 0.5 [\phi ](d) \psi _0 (x,t),
\end{equation}
where
\[
 \psi _0 (x,t) :=\erfc \left( \frac{d(t)-x}{2\sqrt{\ve t}} \right) \quad \hbox{and} \quad \erfc(z) :=\frac{2}{\sqrt{\pi}} \int _{r=z}^\infty e^{-r^2} \ dr.
\]
This function $ y$ satisfies the problem
\begin{subequations}\label{cont-problem}
\begin{eqnarray}
& L  y= 0, \ (x,t) \in  Q, \qquad
 y(x,0)=\left\{ \begin{array}{lll}
\displaystyle
\phi (x), & x < d, \\[1ex]
\displaystyle
\phi (d^-)  , & x = d, \\[1ex]
\displaystyle
\phi (x) -[\phi ](d), & x > d. \end{array}\right.
\\
&y(0,t) =  -0.5 [\phi ](d)  \psi _0 (0,t), \quad 0 < t \leq T,
\\ &  y(1,t) =  -0.5 [\phi ](d)  \psi _0 (1,t), \quad  0 < t \leq T. \label{right-bc}
\end{eqnarray}
\end{subequations}
As $d=O(1)$ we see that all time derivatives of the left boundary condition $ y(0,t)$ are uniformly bounded.
To begin, we shall assume that the final time $T$ is constrained as follows: there exists some $\delta > 0$ such that
\begin{equation}\label{assump-temp}
d(T) \leq 1-\delta .
\end{equation}
In \S \ref{sec:LayerInteract},  we explain what modifications are required in the numerical method when this constraint is not applied.
Note that, by assuming (\ref{assump-temp}), all time derivatives of the right boundary condition $ y(1,t)$ are uniformly bounded (w.r.t. $\ve$) for all $t \leq T$.

From the Appendix, we have the expansion (\ref{expansion})
\[
 y= 0.5 \sum _{i=1}^4 [\phi ^{(i)}](d) \frac{(-1)^i}{i!}   \psi _i (x,t) +   R(x,t), \quad \psi _i \in  C^{i-1+\gamma } (\bar {  Q}),
\]
where the weakly singular functions $ \psi _i $ are defined by (\ref{zero}), (\ref{one}) and (\ref{singular-functions}). Bounds on the partial derivatives of the function $ \psi _i (x,t), i=1,2,3,4$ are stated in (\ref{bounds-singular-functions}), (\ref{more-bounds-singular-functions}).
In the case where the convective coefficient $ a$ is independent of the space variable, we have that
\begin{eqnarray*}
 L  \psi _i =0, \quad  0\le i \le 4  \ \hbox{ and } \\   L  R(x,t)=0, \ (x,t) \in   Q,\  R(x,t) \neq 0, \ (x,t) \in \partial  Q.
\end{eqnarray*}
As $ R(x,0) \in C^4(0,1)$ we also have
\[
  R \in C^{4+\gamma}(\bar{ Q}).
\]
Following the arguments in {\cite[Theorem 1]{dervs-parabolic-conv-diff}}, we have the following decomposition of the smooth remainder
\[
 R =  v_R +  w
\]
into a regular component $ v_R$ and a boundary layer component $ w$. In addition, assuming (\ref{assump-temp}),
 we have the following bounds
\begin{subequations}\label{bnds-comp}
\begin{align}
 \left \Vert \frac{\partial ^{i+j}  v_R }{\partial x ^i \partial t ^j }   \right \Vert _{\bar{ Q}}& \leq  C,\  0 \leq i+j \leq  2, \qquad
\left \Vert \frac{\partial ^{3}  v_R }{\partial x ^3  }  \right \Vert _{\bar{ Q}} \leq  \frac{C}{\ve}, \label{bnds-vR} \\
\left \vert \frac{\partial ^{i+j} }{\partial x ^i \partial t ^j }  w(x,t)\right \vert  & \leq  C\ve ^{-i}(1+\ve ^{1-j})  e^{-\alpha(1-x)/\ve},\quad
0 \leq i+2j \leq  4. \label{bnds-w}
\end{align}
\end{subequations}
Then, combining these bounds with the bounds on the singular functions $ \psi _i$ (see (\ref{bounds-singular-functions}), (\ref{more-bounds-singular-functions}) in  the Appendix for details), we have, assuming (\ref{assump-temp}),
\begin{subequations}\label{bnds-regular}
 \begin{align}
  y  &=  v +  w - 0.5  [\phi'](d)  \psi _1 (x,t),\\
 \hbox{where} \qquad
 v&:=  v_R + 0.5 \sum _{i=2}^4 [\phi ^{(i)}](d) \frac{(-1)^i}{i!}  \psi _i (x,t) \
\\
\hbox{and} \qquad    \left \Vert \frac{\partial ^{i+j}  v}{\partial x ^i \partial t ^j }   \right \Vert   _{\bar{ Q}} &\leq   C, \quad 0 \leq i+2j \leq  2, \\
 \left \vert \frac{\partial ^{2}  }{ \partial t ^2 }  v (x,t)\right \vert    \leq   C\left (1+\frac{\ve}{t}\right), &
\left \vert \frac{\partial ^{3}  }{\partial x ^3  }  v (x,t)\right \vert   \leq  C \left(\frac{1}{\ve}+\frac{1}{\sqrt{\ve t}} \right).
\end{align}
\end{subequations}

\begin{remark}
If $[\phi'](d)=0$, then the function $ y$ is decomposed  simply as $ y  =  v +  w $ and this will have an influence on the order of convergence of the numerical scheme~\eqref{discr-probl}. In this case, it is proved in Theorem~\ref{th_a(t)} that the method converges with almost first order. Otherwise, when $[\phi ' ](d)\neq 0$, the error bound is dominated by the term $ CM^{-1/2}$, which corresponds to the result in \cite{shishkin4}.
\end{remark}

\section{Numerical method and associated error analysis}

 In this section  problem~\eqref{cont-problem} is approximated using  the backward  Euler method and standard central differences on a Shishkin mesh~\cite{fhmos}. Global parameter-uniform error bounds are proved for this scheme.
 Two cases are considered in our error analysis: In \S \ref{sec:LayerNoInteract} the interior layer does not interact with the boundary layer at $x=1$ but it does interact in \S \ref{sec:LayerInteract}.

\subsection{Interior and boundary layers do not interact with each other} \label{sec:LayerNoInteract}

 In this section we assume that~\eqref{assump-temp} is satisfied. Then, the interior layer emanating from $x=d$ and travelling along the characteristic $\Gamma^*$ does not interact with the boundary layer in the vicinity of $x=1.$

Let  $N$ and $M=O(N)$ be two positive integers. We approximate problem~\eqref{problem2} with a finite difference scheme on a mesh  $ \bar Q^{N,M}=\{x_i\}^N_{i=0} \times \{t_j\}_{j=0}^M $.  We denote by $\partial Q^{N,M}:=\bar Q^{N,M}\backslash Q.$
The mesh $\bar Q^{N,M}$ incorporates a uniform mesh  ($t_j:=k j$ with $k=T/M$) for the time variable and a piecewise-uniform mesh for the space variable with $h_i:= x_i-x_{i-1}$.  The piecewise uniform mesh $\{x_i\}^N_{i=0}$ is a Shishkin mesh \cite{fhmos} which splits the interval $[0,1]$ into the two subintervals
\[
[0,1-\sigma]\cup [1-\sigma ,1], \quad \hbox{where} \quad \sigma : =\min \left \{0.5, \frac{\vr}{\alpha}  \ln N \right\}.
\]
The $N$ space mesh points are distributed in the ratio $N/2:N/2$ across the two subintervals.
The discrete problem\footnote{We use the following notation for the finite difference approximations of the derivatives:
\begin{align*}
D^-_t Y (x_i,t_j) := \ds\frac{Y(x_i,t_j)-Y(x_i,t_{j-1})}{k}, \quad
D^-_x Y(x_i,t_j) :=\ds\frac{Y(x_i,t_j)-Y (x_{i-1},t_j)}{h_i}, \\
D^+_x Y (x_i,t_j)  :=\ds\frac{Y(x_{i+1},t_j)-Y(x_i,t_j)}{h_{i+1}}, \
 \delta^2_x Y(x_i,t_j) := \ds\frac{2}{h_i+h_{i+1}}(D_x^+Y(x_i,t_j)-D^-_x Y(x_i,t_j)).
 \end{align*}}
 is: Find $Y$ such that
\begin{subequations}\label{discr-probl}
\begin{align}
L^{N,M} Y &:=  -\ve \delta ^2_x  Y  + a D^-_x Y + D^-_t Y =0,\quad t_j >0,  \\
Y (x_i,0) & = y (x_i,0), \ 0 <x_i<1,\   Y(0,t_j) =   Y(1,t_j) = 0,\ t _j \geq 0.
\end{align}
\end{subequations}

\begin{remark}\label{method-general}
The method and the analysis presented here for problem (\ref{problem2}) can be easily extended to a wider class of problems.  Consider the more general problem
\begin{subequations}\label{problem-general}
\begin{align}
 &L_g  u := -\ve {u}_{xx} +  a(t) {u}_x +  b(t) {u}+{u}_t= f, \ (x,t) \in  Q , \\
& u =  g \in \partial Q:=\bar Q\setminus Q, \quad [u](d,0) \neq 0;\\
&b(t) \geq 0,\  0 \leq t \leq T.
\end{align}
\end{subequations}
As the problem is linear we can write the solution  as the sum
\begin{align*}
 u &= e^{-\int _0^t b(r) dr} u_h +  u_p,   \quad \hbox{where}
\\
 L  u_h & =0 ,\qquad   L_g  u _p= f, (x,t) \in  Q; \\
 u_h(x,0) & = \phi (x);\quad  u_p(x,0) = u(x,0) -\phi (x);\\
 u_h(0,t) & =  u_h(1,t) =0; \quad  u_p(0,t) =  u(0,t),\  u_p(1,t) =  u(1,t).
\end{align*}
In addition to the constraint (\ref{1d}),  we assume that $b,f$ are sufficiently regular and that sufficient compatibility is imposed at the points $(0,0), (1,0)$ so that $u_p \in C^{4+\gamma}(\bar Q)$.  The component $u_h$ satisfies problem (\ref{problem2}) and it will examined below. The component $ u_p$ can be decomposed into a regular and boundary layer component as in  \cite{grey}.
For problem (\ref{problem-general}), we first subtract off the term
\[
 y (x,t) :=  u(x,t) - S_g(x,t), \qquad S_g(x,t):= 0.5 [\phi ](d) e^{-\int _0^t b(r) dr}\psi _0 (x,t),
\]
then $
L_gy=f, (x,t) \in Q$ and $y=u-S_g,\ (x,t) \in \partial Q$.
The corresponding discrete problem is then given by
\begin{subequations}\label{discr-problem-general}
\begin{align*}
L_g^{N,M} Y &:=  -\ve \delta ^2_x  Y  + a D^-_x Y +bY +  D^-_t Y =f,\quad (x_i,t_j) \in Q^{N,M},  \\
Y (x_i,0) & = y (x_i,0), \quad (x_i,t_j) \in \partial Q^{N,M};
\end{align*}
\end{subequations}
where the mesh  $Q^{N,M}$ is as described earlier.
\end{remark}

We form a global approximation $\bar Y$ using simple  bilinear interpolation:
\[
\bar Y(x,t) : = \sum _{i=0,j=1}^{N,M} Y(x_i,t_j) \varphi _i(x)  \eta_j(t)
\]
where $\varphi _i(x)$ is the standard hat function centered at $x=x_i$ and $ \eta _j(t) :=(t-t_{j-1})/k, \, t \in [t_{j-1},t_j), \,  \eta _j(t) :=0, \, t \not \in [t_{j-1},t_j)$.

\begin{theorem} \label{th_a(t)} Assume (\ref{assump-temp}) and $M=O(N)$.
If $Y$ is the solution of (\ref{discr-probl}) and $y$ is the solution of (\ref{cont-problem}), then
\begin{eqnarray*}
\Vert  \bar Y- y \Vert _{\bar Q} \leq CN^{-1}\ln N
+ C\vert [\phi ' ](d) \vert  M^{-1/2}.
\end{eqnarray*}
\end{theorem}
\begin{proof}
As in the case of the continuous problem, the discrete solution can be decomposed into the sum $Y = V+W- 0.5\ [\phi ' ](d)\Psi$, where
\begin{align*}
L^{N,M}V &=Lv,\ (x_i,t_j) \in {Q^{N,M}} \quad \hbox{and} \quad V= v,\ ({ x_i},t_j) \in  { \partial Q^{N,M};} \\
 L^{N,M}W &=0,\ (x_i,t_j) \in { Q^{N,M}} \quad \hbox{and} \quad  W= w,\ ({x_i},t_j) \in  { \partial Q^{N,M};} \\
  L^{N,M}\Psi &=0,\ (x_i,t_j) \in { Q^{N,M}} \quad \hbox{and} \quad  \Psi= \psi _1,\ ({ x_i},t_j) \in  { \partial Q^{N,M}.}
\end{align*}
Using the bounds on the derivatives (\ref{bnds-w}) of the component $ w$, truncation error bounds, discrete maximum principle, a suitable discrete barrier function and following the arguments in \cite{ria}, we can establish the following bounds
\begin{equation} \label{error-w}
 \vert ( w- W)(x_i,t_j) \vert  \leq  C N^{-1} \ln N  + CM^{-1}, \quad   (x_i,t_j)\in\bar Q^{N,M}.
\end{equation}
We next bound the error due to the regular component $ v$.
Note that if the truncation error is denoted by ${\cal T}_{i,j}:= L^{N,M} ( v -V)(x_i,t_j)$, then
\begin{align*}
 \vert {\cal T}_{i,j} \vert  &\leq C \ve (h_i+h_{i+1}) \left\Vert \frac{\partial ^3  v (x,t_j)}{\partial x^3} \right\Vert _{(x_{i-1},x_{i+1}) } \\
 &+C \min \left\{ h_i\left\Vert \frac{\partial ^2  v(x,t_j)}{\partial x^2} \right\Vert _{(x_{i-1},x_{i}) }, \left\Vert \frac{\partial   v(x,t_j)}{\partial x} \right\Vert _{(x_{i-1},x_{i}) }  \right \}
 \\
 &+
C\min \left\{ \frac{1}{k} \int _{w=t_{j-1}}^{t_j} \int _{r=w}^{t_j}  \left\vert \frac{\partial ^2  v(x_i,r)}{\partial t^2} \right\vert dr \ dw ,    \left\Vert \frac{\partial   v(x_i,t)}{\partial t} \right\Vert _{  (t_{j-1}, t_j )}  \right\},
\end{align*}
as
\[
\vert D_t^-  v(x_i,t_j) \vert  \leq \frac{1}{k} \int _{r=t_{j-1}} ^{t_j} \left \vert \frac{\partial  v(x_i,r)}{\partial r}   \ dr \right \vert \leq C \left \Vert \frac{\partial   v(x_i,t)}{\partial t} \right \Vert _{  (t_{j-1}, t_j )}.
\]
Hence, using the bounds (\ref{bnds-regular}) on the derivatives of $ v$,   we obtain
\begin{subequations}\label{error-TE-v}
\begin{align}
\vert {\cal T}_{i,1} \vert &  \leq C  \qquad \hbox{and} \\
\vert {\cal T}_{i,j} \vert &  \leq  C \left (1+\frac{\sqrt{\vr}}{\sqrt{t_j}}\right) N^{-1}  + C k \left \Vert \frac{\partial^2  v (x_i,t)}{\partial t^2} \right \Vert _{ (t_{j-1}, t_j )}  \nonumber \\
& \leq  C \left (1+\frac{\sqrt{\vr}}{\sqrt{t_j}}\right) N^{-1}   + CM^{-1} \left(1+\frac{\vr }{t_{j-1}}\right),\quad t_j > t_1.
\end{align}
\end{subequations}
We now mimic  the argument in \cite{Zhemukhov1} and note that at each time level,
\[
\left(-\ve \delta ^2_x    + a D^-_x  +\frac{1}{k}I \right)  ( v -V)(x_i,t_j) =  {\cal T}_{i,j} + \frac{1}{k}  ( v -V)(x_i,t_{j-1}), \ t_j >0 .
\]
From this  and~\eqref{error-TE-v}, we  deduce the error bound
\begin{align}
\vert  ( v -V)(x_i,t_j) \vert &\leq C k \sum _{n=1}^j \vert {\cal T}_{i,n} \vert  \leq C M^{-1}+ k \sum _{n=2}^j \vert {\cal T}_{i,n} \vert \nonumber \\
&{ \leq C (N^{-1} +M^{-1})+C M^{-3/2}\sqrt{\vr} \int_{s=1}^j \frac{ds}{\sqrt{s}}+C M^{-1}\ve \int_{s=1}^j\frac{ds}{s} } \nonumber\\
&\leq   C N^{-1} +C M^{-1} (1+ \ve \ln (1+j)) \nonumber \\
& \leq CN^{-1}+ CM^{-1} (1+\ve  \ln M). \label{error-v}
\end{align}

Finally, we consider  the error due to the weakly singular component $ \psi _1$.  Its truncation error is denoted by ${\cal \widetilde T}_{i,j}:= L^{N,M} ( \psi_1 -\Psi)(x_i,t_j)$. The argument splits into the two cases of $\ve \leq CM^{-1}$ and $\ve \geq CM^{-1}$.
In the first case, where $\ve \leq CM^{-1}$, using the bound (\ref{bound-C}) on the first  time derivative of $ \psi _1$   at $t=t_1$, we have
\begin{align}
 \vert  {\cal \widetilde T}_{i,1} \vert &\leq  C\frac{  N^{-1} }{k} + C   a(t_1) + C\frac{1}{k} \int _{r=0} ^{k} \left(1 +\sqrt {\frac{\ve}{r}}\right) \ dr \nonumber \\
&\leq  C\frac{  N^{-1}  }{k} + C. \label{error-TE-Psi0-1}
\end{align}
For $t_j>t_1,$
we first sharpen our bounds on the function $\psi _1$. Observe that
\begin{align*}
& \vert \psi _0(x,t) \vert \leq C E(x,t), \qquad \hbox{if} \quad x \leq d(t), \\
& \vert \psi _1(x,t) -2(d(t)-x)) \vert = \left \vert (d(t)-x)( \psi _0 -2) -2\frac{\sqrt{\ve t}}{\sqrt{\pi}} E    \right \vert
\\
 & \hspace{1.5cm} \leq CE(x,t), \qquad \hbox{if} \quad x \geq d(t).
\end{align*}
Hence, instead of (\ref{bound-C}), (\ref{bound-E}) on the first derivatives of $\psi _1$, we have the following derivative bounds:
\begin{align*}
\left \vert \frac{\partial  }{\partial t  }  \psi _1(x,t)\right \vert &\leq   C  \left(1 +\sqrt {\frac{\ve}{ t}}\right)  E (x,t), \ x \leq d(t), \\
\left \vert \frac{\partial  }{\partial t  }  ( \psi _1(x,t) -2(d(t)-x))\right \vert &\leq   C  \left(1 +\sqrt {\frac{\ve}{ t}}\right)  E (x,t), \ x \geq d(t), \\
 \left \vert \frac{\partial  }{\partial x  } \psi _1(x,t) \right \vert    &\leq  C E_\gamma (x,t), \  x \leq d(t), \\
\left \vert \frac{\partial  }{\partial x  }( \psi _1(x,t) -2(d(t)-x))\right \vert    &\leq C E_\gamma (x,t), \  x \geq d(t).
\end{align*}
When estimating the truncation error due to the presence of $\psi _1$, in the case of $\ve \leq CM^{-1}$, we have at each time level
\begin{align*}
(L-L^N)\psi _1 &= (L-L^N)(\psi _1 -2(d(t)-x))  - \frac{2}{k} \int _{s=t_{j-1}}^{t_j} (a(s)-a(t)) \, ds \\
&=  (L-L^N)(\psi _1 -2(d(t)-x)) +CM^{-1}.
\end{align*}
Depending on where $x_i$ is located relative to $d(t_j)$ we can bound the truncation error
using the first space derivative for the term corresponding to the convective term $  (\psi_1)_x$, to get
 \begin{align}
& \vert {\cal \widetilde T}_{i,j} \vert \leq C\frac{  N^{-1} }{t_j} + C \max _{x\in (x_{i-1},x_i)} e^{-\frac{(x-d(t_j))^2}{4\ve t_j}}+ C\max _{t\in (t_{j-1},t_j)} e^{-\frac{( x_i-d(t))^2}{4 \ve t}} +CM^{-1} \nonumber\\
&\quad \leq C\frac{ N^{-1}}{t_j} + C \max _{x\in (x_{i-1},x_i)} e^{-\frac{CM(x-d(t_j))^2}{T}}+ C\max _{t\in (t_{j-1},t_j)} e^{-\frac{CM( x_i-d(t))^2}{T}} +CM^{-1}. \label{error-TE-Psi0}
\end{align}
Let $x_i$ be fixed. Over the interval $(x_{i-1},x_i)$,
\[
\max _{x\in (x_{i-1},x_i)} e^{-\frac{CM(x-d(t_j))^2}{T}} = \begin{cases} e^{-\frac{CM( x_{i-1}-d(t_j))^2}{T}}, & \hbox{ if }  d(t_j) < x_{i-1},  \\
1, & \hbox{ if } d(t_j) \in [x_{i-1},x_i], \\
e^{-\frac{CM( x_{i}-d(t_j))^2}{T}}, & \hbox{ if } d(t_j)  > x_i.
\end{cases}
\]
Let $d(t_n)$ be the first time $t_n$ for which $x_{i-1} < d(t_n)$. If $d(t_n) < x_i$, then for some $m >n$
\[
d(t_m) =d(t_n) + \int _{r=t_n}^{t_m} a(r) \ dr > x_{i-1}+ \alpha (m-n) k,
\]
where we have used that $d'(t) =a(t) \geq \alpha >0$.
As $M=O(N)$, there will  be, at most, a finite number (independent  of $N$) time points for which $d(t_j) \in [x_{i-1},x_i]$.
Noting $\int _{r=-\infty}^\infty \frac{1}{\sqrt{k}}e^{-\frac{r^2}{k}} \ dr = \sqrt{\pi }$ we have, for each fixed $x_i$,
 \begin{align} \label{error-TE-Psi0-sum}
\sum_{j=1}^M e^{-\frac{CM(x_i-d(t_j))^2}{T}}  =\frac{1}{\sqrt{k}} \sum_{j=1}^M \frac{k}{\sqrt{k}}e^{-\frac{C(x_{i}-d(t_j))^2}{k}}  \le C M^{1/2}.
\end{align}
From~\eqref{error-TE-Psi0-1}, \eqref{error-TE-Psi0} and~\eqref{error-TE-Psi0-sum}, one has for $\ve \leq CM^{-1}$,
\begin{equation} \label{error-psi0}
\vert (\psi_1-\Psi)(x_i,t_j)\vert \le  \sum _{n=1}^j k \vert {\cal \widetilde T}_{i,n} \vert  \leq C  N^{-1} \ln (1+j)+ CM^{-1/2}.
\end{equation}

In the second case, where $\ve \geq CM^{-1}$,  at the first time level we have from (\ref{bound-C})
\begin{align*}
 \vert {\cal \widetilde T}_{i,1} \vert &\leq C\frac{  N^{-1} }{k}+ C\frac{N^{-1}}{\sqrt{\ve k}} + C\frac{1}{k} \int _{r=0} ^{k} \left(1 +\sqrt {\frac{\ve}{r}}\right) \ dr \\
&\leq  C\frac{  N^{-1} }{k} + C  \sqrt{\ve} M^{1/2};
\end{align*}
and at all the other time levels, from \eqref{bound-D} and~\eqref{bound-E}, we have for $j>1$
\begin{align*}
 \vert {\cal \widetilde T}_{i,j} \vert &\leq C\frac{ N^{-1} }{t_j} + C \frac{N^{-1}}{\sqrt{\vr t_j} }+ C \int _{r=t_{j-1}} ^{t_j} \frac{1}{r} \ dr \\
&  + C \sqrt{\ve} \int _{r=t_{j-1}} ^{t_j} \sqrt {\frac{1}{r^3}} \ dr +C  \frac1{\sqrt{\vr}} \int _{r=t_{j-1}} ^{t_j} \frac{1}{ \sqrt{r}} \ dr\\
& \le  C \frac{N^{-1}}{t_j} + C  \int _{s={j-1}} ^{j} \frac{1}{ s} \ ds \\
&+ CM^{1/2}\sqrt{\ve} \int _{s={j-1}} ^{j}  {\frac1{\sqrt{s^3}} \ ds
+C \frac{M^{-1/2}}{\sqrt{\vr}} \int _{s={j-1}} ^{j} \frac{1}{ \sqrt{s} }\ ds}.
\end{align*}
Applying the earlier argument, we have for $\ve \geq CM^{-1}$,
\begin{align}
&  \vert (\psi_1-\Psi)(x_i,t_j)\vert \le  \sum _{n=1}^j k \vert {\cal \widetilde T}_{i,n} \vert \leq CN^{-1}+ C \sqrt{\vr} M^{-1/2} +  \sum _{n=2}^j k \vert {\cal \widetilde T}_{i,n} \vert  \nonumber \\
& \quad \leq CM^{-1/2}+ C (N^{-1}+M^{-1}) \int_{s=1}^j \frac{ds}{s}+{C M^{-1/2}\sqrt{\ve} \int_{s=1}^j \frac{ds}{\sqrt{s^3}}} \nonumber \\
& \quad + C M^{-1} \int_{s=1}^j \frac{ds}{\sqrt{s}} \nonumber \\
& \quad \leq C (N^{-1}+M^{-1}) \ln (1+j) + CM^{-1/2}. \label{error-psi}
\end{align}
Hence,  if $M=O(N)$,  from~\eqref{error-w}, \eqref{error-v}, \eqref{error-psi0} and~\eqref{error-psi}, we have the nodal error estimate
\begin{eqnarray*}
  \vert (Y-y)(x_i,t_j) \vert \leq C  N^{-1} \ln N + C\vert [\phi ' ](d) \vert  M^{-1/2}.
\end{eqnarray*}

Combine the arguments in \cite[Theorem 3.12]{fhmos} with the interpolation bounds in \cite[Lemma 4.1]{styor4} and the bounds on the derivatives of the components $ v,  w, \psi _1 $. Note that from \cite[Lemma 4.1]{styor4}, we only require the first time derivative of any component of  $ y$ to be uniformly bounded. For the weakly singular component $ \psi _1 $, the argument is split into the two cases of $\ve \leq CM^{-1}$ and $\ve \geq CM^{-1}$.
\end{proof}

\begin{remark} \label{rem:improvement}
The error estimates of Theorem~\ref{th_a(t)} reveal that the method~\eqref{discr-probl} converges with order $M^{-1/2}$ when $[\phi'(d)]\ne 0$. In order to increase the rate of convergence, the analytical/numerical method can be used to approximate the component $y$. Thus, from the expansion~\eqref{expansion}, one can consider the decomposition
\begin{align*}
u(x,t)&=0.5 [\phi](d) \psi_0(x,t)+y(x,t) \\
& =0.5 [\phi](d) \psi_0(x,t)-0.5[\phi'(d)] \psi_1(x,t) +y_1(x,t).
\end{align*}
In   Example~\ref{ex2} in \S 4, we observe an improvement in the orders of convergence when $y_1$ is approximated with the numerical scheme~\eqref{discr-probl} instead of $y$.
\end{remark}

\subsection{Interior and boundary layers interact with each other} \label{sec:LayerInteract}

In the case where (\ref{assump-temp}) is not assumed,
the bounds (\ref{bnds-vR}) on $ v_R$ are still applicable, but we need to determine alternative  bounds to  (\ref{bnds-w}), on the boundary layer function $w$. The boundary layer function  is the solution of the problem
\begin{align*}
 L  w &=0, (x,t) \in  Q; \qquad   w(x,0) =0; \quad 0 \leq x \leq 1\\
  w(0,t) &=0; \quad t >0,  \\
   w(1,t &)=
 0.5\left( [\phi ](d) - \sum _{i=0}^4 [\phi ^{(i)}](d) \frac{(-1)^i}{i!}   \psi _i (1,t)\right)  -    v_R(1,t), \quad t >0.
\end{align*}
Since
\begin{equation} \label{psi0t}
\frac{\partial   \psi _0}{\partial t} (1,t)=  \frac{1}{\sqrt{\ve \pi t}}  \Bigl(\frac{(d(t)-1)-2ta}{2 t} \Bigr) e^{-\frac{(1-d(t))^2}{4\ve t}}
\end{equation}
 when (\ref{assump-temp}) is not satisfied, then there exists  a $T_1$ (independent of $\ve $), with $0<T_1 <T$ such that
\[
1 \geq \frac{1-d(t)}{1-d} \geq \delta >0, \quad \hbox{for} \quad t \leq T_1.
\]
Then, for $t \leq T_1$
\[
\left \vert \frac{\partial   \psi _0}{\partial t} (1,t) \right \vert \leq   \frac{C}{\sqrt{\ve  t}} e^{-\frac{\delta^2 (1-d)^2}{4\ve t}} \leq   \frac{C}{\delta (1-d)} \leq C
\]
and for $t > T_1$
\[
\left \vert \frac{\partial   \psi _0}{\partial t} (1,t) \right \vert \leq   \frac{C}{\sqrt{\ve  T^3_1}} { e^{-\frac{(1-d(t))^2}{4\ve t}}}  \leq   \frac{C}{\sqrt{\ve }}  e^{-\frac{(1-d(t))^2}{4\ve t}} .
\]
In the same way, we can establish  that
\[
\Bigl \vert \frac{\partial ^{j}   w}{ \partial t ^j } (1,t)\Bigr \vert  \leq  C \left(1+\ve ^{-j/2}  e^{-\frac{(1-d(t))^2}{4\ve t}}\right), \quad j=1,2.
\]
Based on the argument in \cite[Theorem 1]{dervs-parabolic-conv-diff}  one can deduce the following bounds
\begin{eqnarray*}
\Bigl \vert \frac{\partial ^{i+j}  w}{\partial x ^i \partial t ^j } (x,t)\Bigr \vert  \leq  C\ve ^{-i}(1+\ve ^{-j/2}) e^{-\alpha(1-x)/\ve}, \quad 0 \leq i+2j \leq  4.
\end{eqnarray*}
In the coarse mesh $x_i \leq 1-\sigma$,
\[
\vert   w (x_i,t_j)  \vert \leq CN^{-1}, \quad \vert W (x_i,t_j) \vert \leq CN^{-1},
\]
and the truncation error within the fine mesh region $ (1-\sigma ,1) \times (0,T]$  is of the form
\[
\vert L ^{N,M}(  w -  W)(x_i,t_j) \vert \leq C\frac{N^{-1} \ln N +M^{-1}}{\ve}.
\]
Use a discrete barrier function \cite[Theorem 2]{dervs-parabolic-conv-diff}  to deduce that
\[
\vert  ( w -  W)(x_i,t_j) \vert \leq C(N^{-1} \ln N +M^{-1}).
\]
Hence, the nodal error bound in Theorem \ref{th_a(t)} still applies
in the case where (\ref{assump-temp}) is not assumed. To extend this nodal error bound to a global error bound,   we first observe that (if (\ref{assump-temp}) is violated),
 then there exists a $T_* <T$ such that
\[
 d(T_*)=d+\int _{s=0}^{T_*} a(s) \ ds =1 \]
 and $T_*=O(1)$ as $T_* \geq (1-d)/\Vert a \Vert$.
 From~\eqref{psi0t}, note also that, for $t=O(1)$
\[
\frac{\partial  \psi _0}{\partial t} (1,t)=  \frac{C}{\sqrt{\ve }} e^{-\frac{(1-d(t))^2}{4\ve t}},
\]
and from (\ref{right-bc})  one has for $j=1,2$
\[
\Bigl \vert \frac{\partial ^j }{\partial t ^j }   y(1,t)\Bigr \vert \leq C \ve ^{-j/2},\quad \hbox{when} \quad  \vert t-T_* \vert \le C\sqrt{\ve \ln (1/\ve)}.
\]
To interpolate this  layer function along the boundary $x=1$, we need to introduce a Shishkin mesh in time, which places $M/2$ mesh points into the time interval
\begin{equation}\label{time-Sh-mesh}
[ T_* - \tau, T_* + \tau ], \text{ with }  \tau :=\min
\left \{\frac{T_*}{2}, \frac{T-T_*}{2},2\frac{\sqrt{T_* \vr\ln (M)} } {\alpha} \right \}.
\end{equation}
Subdivide each of $[0,T_*-\tau]$ and $[T_*+\tau,T]$ by an equidistant mesh with $M/4$ subintervals.

With this modification to the numerical method, the error bound in Theorem \ref{th_a(t)} applies, as the linear interpolant of $ y(1,t)$  with $t\in(t_{j-1},t_j)$ satisfies for $ t \not \in \ [ T_* - \tau, T_* + \tau ]$ and $\tau = O(\sqrt{\ve \ln M})$
\begin{align*}
 \vert  y(1,t) -  y_I(1,t) \vert &\leq  C  \vert  \psi _0(1,t) \vert \leq C e^{-\frac{(d(T^*)-d(t))^2}{4\ve t}}\\
&\leq  C e^{-\frac{(\int _t^{T^*} a(s) ds)^2}{4\ve t}} \leq Ce^{-\frac{(\alpha (T^*-t) )^2}{4\ve t}} \\
&\leq  Ce^{-\frac{(\alpha \tau )^2}{4\ve T_*}}
\leq CM^{-1}
\end{align*}
and for $ t \in \ [ T_* - \tau, T_* + \tau ]$
\begin{eqnarray*}
 \vert  y(1,t) -  y_I(1,t) \vert &\leq &
C (t_j-t_{j-1}) \Bigl \Vert \frac{\partial  }{\partial t }  \psi _0(1,t) \Bigr\Vert_{[t_{j-1},t_j]}
 \leq
CM^{-1}\sqrt{\ln (M) }.
\end{eqnarray*}

\section{Numerical results}

In this section we present   the numerical results for five test examples whose solutions are unknown. The global orders of convergence are estimated using the two-mesh method~\cite[Chapter 8]{fhmos}. In this section the computed solutions with~\eqref{discr-probl} on the Shishkin meshes $\bar Q^{N,M}$ and $\bar Q^{2N,2M}$ will be denoted, respectively,  by $Y^{N,M}$ and $Y^{2N,2M}$.
   Let $\bar Y^{N,M}$  be the bilinear interpolation of the discrete solution $Y^{N,M}$  on the mesh $\bar Q^{N,M} $. Then, compute the maximum two-mesh global differences
$$
D^{N,M}_\ve:= \Vert \bar Y^{N,M}-\bar Y^{2N,2M}\Vert _{\bar Q^{N,M} \cup \bar Q^{2N,2M}}
$$
and use these values to estimate the orders of global  convergence $ P^{N,M}_\ve$
$$
 P^{N,M}_\ve:=  \log_2\left (\frac{D^{N,M}_\ve}{D^{2N,2M}_\ve} \right).
$$
The uniform  two-mesh global differences $D^{N,M}$ and the uniform orders of global convergence $ P^{N,M}$ are calculated by
$$
D^{N,M}:= \max_{\ve \in S} D^{N,M}_\ve, \quad  P^{N,M}:=  \log_2\left ( \frac{D^{N,M}}{D^{2N,2M}} \right),
$$
where $S=\{2^0,2^{-1},\ldots,2^{-26}\}$. For each of the five test examples, plots of $\bar Y^{N,M}$ and $\bar U^{N,M}:=\bar Y^{N,M}+ \bar S$ (see (\ref{def-S}))
are given for the sample values of $\ve = 2^{-12}$ and $N=M=64.$

In \S \ref{sec:numer-theory} the numerical results for three representative examples are given and they indicate that the error bounds established in Theorem~\ref{th_a(t)} are sharp. In the first two examples the interior layer does not interact with the boundary layer but in the third example they do interact. The  two examples considered in \S \ref{sec:numer-extension} are not covered by the theory developed in earlier sections.

\subsection{Test problems covered by the theory in Theorem~\ref{th_a(t)}} \label{sec:numer-theory}
\begin{example} \label{ex1}
 Consider the following  test problem
\[
\begin{array}{l}
-\ve  u_{xx} + (1+t^2) u_x + u_t= 4x(1-x)t+t^2, \quad (x,t) \in (0,1)\times (0,0.5], \\
 u(x,0)=-2, 0 \leq x < 0.3, \quad    u(x,0)=1, \ 0.3 \leq x \leq 1,         \\
 u(0,t)=-2, \quad    u(1,t) =1, \ 0 < t \leq 0.5.
\end{array}
\]
\end{example}

Note  that in this example $\ [ \phi  ](0.3) = 3,$ $[\phi ' ](0.3) = 0$  and the characteristic curve $\Gamma^*$ is $d(t)=t+t^3/3+0.3$. The computed approximation $Y$ with the scheme~\eqref{discr-probl} and the numerical solution $U$ are displayed in Figure~\ref{fig:ex1-solutions}, where we can observe that the interior and boundary layer do not merge.

 In the last row of Table~\ref{tb:ex1} and all subsequent  tables in this paper, the uniform two-mesh global differences and their orders of convergence are provided.
The results displayed in Table~\ref{tb:ex1} agree with the
 theoretical error estimates established in Theorem~\ref{th_a(t)}.

 \begin{figure}[h!]
\centering
\resizebox{\linewidth}{!}{
	\begin{subfigure}[Approximation to $y$]{
		\includegraphics[scale=0.5, angle=0]{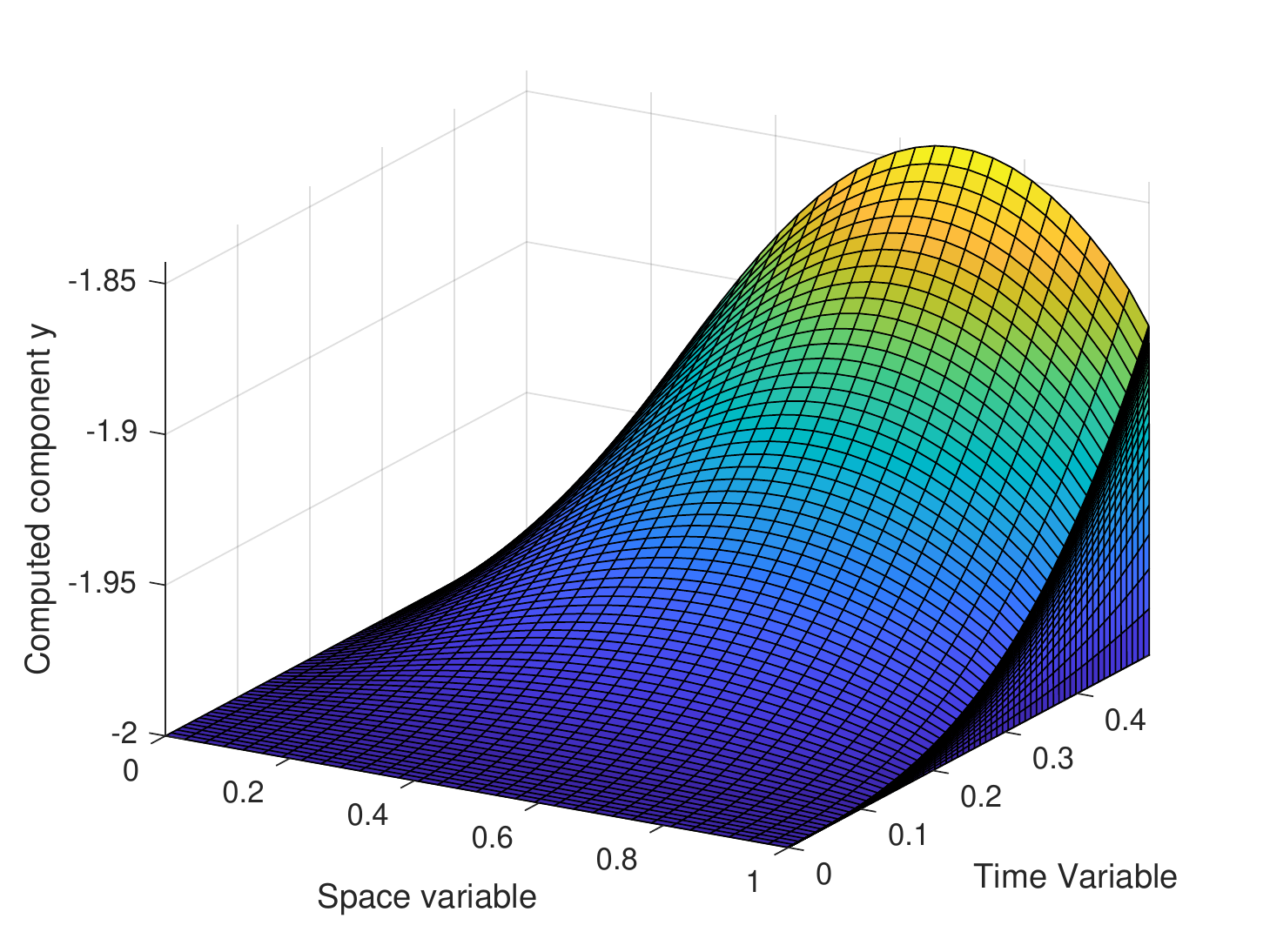}
		}
    \end{subfigure}
\begin{subfigure}[Approximation to $u$]{
		\includegraphics[scale=0.5, angle=0]{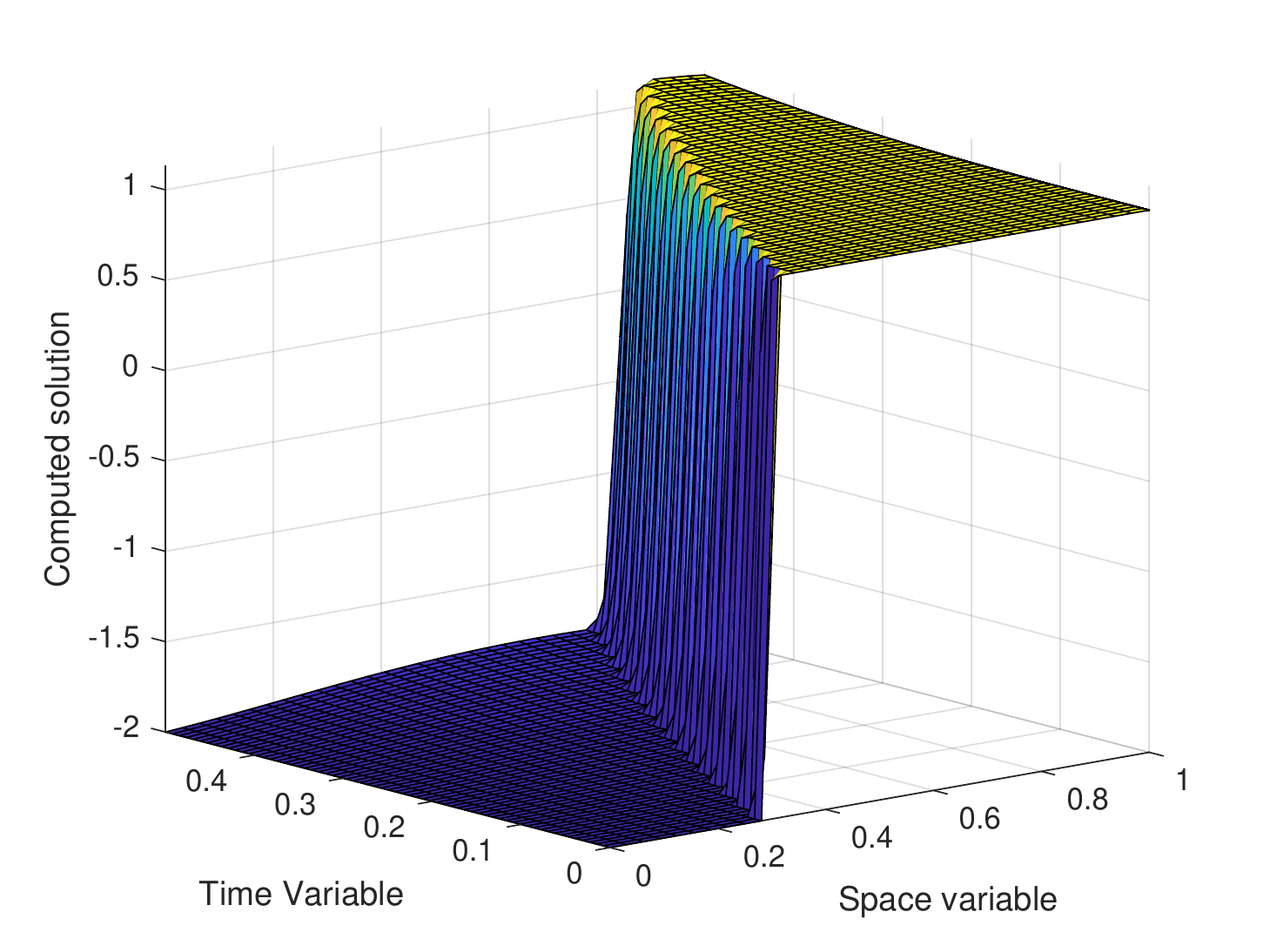}
		}
	\end{subfigure}
}
	\caption{Example~\ref{ex1}: Numerical approximations to $y$ and $u$ with $\vr=2^{-12}$ and $N=M=64$}
	\label{fig:ex1-solutions}
 \end{figure}

\begin{table}[h]
\caption{Example~\ref{ex1}: Maximum two-mesh global differences and orders of convergence for the function $y$}
\begin{center}{\tiny \label{tb:ex1}
\begin{tabular}{|c||c|c|c|c|c|c|c|}
 \hline  & N=M=32 & N=M=64 & N=M=128 & N=M=256 & N=M=512 & N=M=1024 & N=M=2048 \\
\hline \hline $\vr=2^{0}$
&{\bf 3.495E-02} &9.470E-03 &3.704E-03 &1.656E-03 &8.105E-04 &4.011E-04 &1.995E-04 \\
&1.884&1.354&1.161&1.031&1.015&1.007&
\\ \hline $\vr=2^{-2}$
&1.230E-02 &5.783E-03 &2.798E-03 &1.376E-03 &6.821E-04 &3.396E-04 &1.695E-04 \\
&1.088&1.048&1.024&1.012&1.006&1.003&
\\ \hline $\vr=2^{-3}$
&1.673E-02 &{\bf 1.079E-02} &{\bf 5.191E-03} &2.547E-03 &1.262E-03 &6.280E-04 &3.133E-04 \\
&0.633&1.055&1.027&1.013&1.007&1.003&
\\ \hline $\vr=2^{-4}$
&1.368E-02 &7.991E-03 &4.652E-03 &{\bf 2.642E-03} &{\bf 1.484E-03 }&{\bf 8.235E-04} &{\bf 4.726E-04} \\
&0.776&0.781&0.816&0.832&0.850&0.801&
\\ \hline $\vr=2^{-6}$
&4.800E-03 &2.698E-03 &1.557E-03 &8.862E-04 &4.976E-04 &2.763E-04 &1.516E-04 \\
&0.831&0.792&0.813&0.833&0.848&0.866&
\\ \hline $\vr=2^{-8}$
&2.468E-03 &1.323E-03 &7.255E-04 &3.993E-04 &2.204E-04 &1.221E-04 &6.748E-05 \\
&0.900&0.866&0.862&0.857&0.851&0.856&
\\ \hline $\vr=2^{-10}$
&2.875E-03 &1.679E-03 &9.370E-04 &5.184E-04 &2.860E-04 &1.572E-04 &8.599E-05 \\
&0.776&0.841&0.854&0.858&0.864&0.870&
\\ \hline $\vr=2^{-12}$
&2.968E-03 &1.707E-03 &9.543E-04 &5.288E-04 &2.927E-04 &1.612E-04 &8.819E-05 \\
&0.798&0.839&0.852&0.853&0.861&0.870&
\\ \hline $\vr=2^{-14}$
&2.992E-03 &1.714E-03 &9.586E-04 &5.313E-04 &2.943E-04 &1.623E-04 &8.892E-05 \\
&0.804&0.839&0.851&0.852&0.859&0.868&
\\
\hline &\vdots &\vdots &\vdots &\vdots &\vdots &\vdots &\vdots \\
&&&&&&&
\\ \hline $\vr=2^{-26}$
&3.001E-03 &1.717E-03 &9.600E-04 &5.322E-04 &2.949E-04 &1.626E-04 &8.913E-05 \\
&0.806&0.839&0.851&0.852&0.859&0.867&
\\ \hline $D^{N,M}$
&3.495E-02 &1.079E-02 &5.191E-03 &2.642E-03 &1.484E-03 &8.235E-04 &4.726E-04 \\
$P^{N,M}$ &1.696&1.055&0.974&0.832&0.850&0.801&\\ \hline \hline
\end{tabular}}
\end{center}
\end{table}

\begin{example} \label{ex2}
 Consider the following  test problem
\[
\begin{array}{l}
-\ve  u_{xx} + (1+t^2) u_x +  u+ u_t= 4x(1-x)t+t^2, \quad (x,t) \in (0,1)\times (0,0.5], \\
 u(x,0)=-x^3, \ 0 \leq x < 0.3, \quad    u(x,0)=(1-x)^3, \ 0.3 \leq x \leq 1,         \\
 u(0,t)=   u(1,t) =0, \ 0 < t \leq 0.5.
\end{array}
\]
\end{example}
The computed maximum two-mesh global differences and orders of convergence are given in Table~\ref{tb:ex2}. Unlike Example~\ref{ex1}, the initial condition satisfies $[ \phi ' ](0.3)\ne 0$ and its influence on the orders of convergence is clearly shown in this table. The orders of convergence are reduced to a 0.5, which is in agreement with the error bound given in Theorem~\ref{th_a(t)}.

 For this particular example, we show that more accurate approximations to the solution can be obtained if the decomposition given in Remark~\ref{rem:improvement} is considered. The numerical approximation to  $y_1$ and $u$ are displayed in Figure~\ref{fig:ex2-solution+2md-y1}.
 The uniform two-mesh global differences and their orders of convergence are given in Table~\ref{tb:ex2-y1}, showing that the numerical/analytical scheme converges globally and uniformly with almost first order when the component $y_1$ is approximated.

 \begin{figure}[h!]
\centering
\resizebox{\linewidth}{!}{
	\begin{subfigure}[Approximation to $y$]{
		\includegraphics[scale=0.5, angle=0]{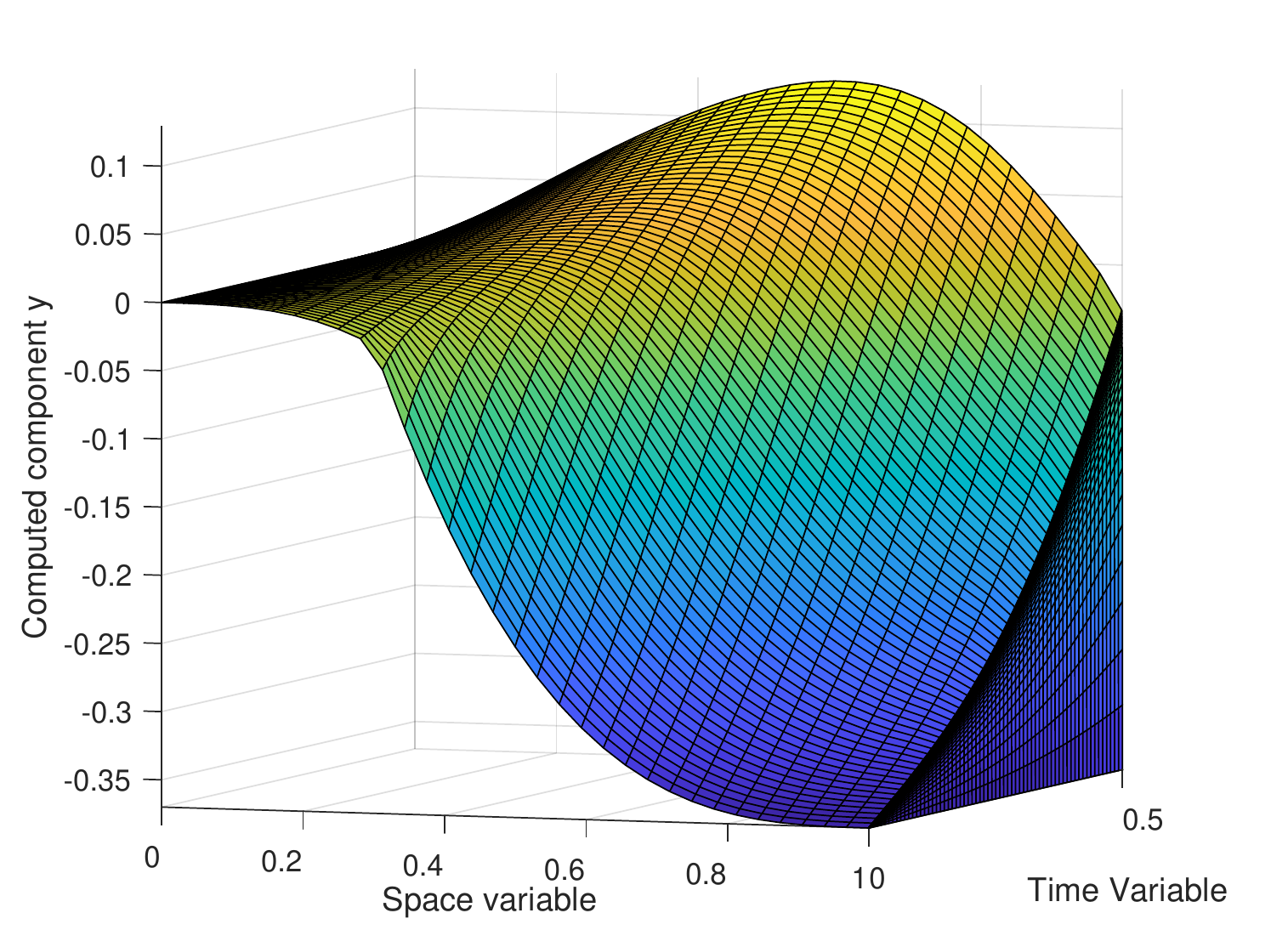}
		}
    \end{subfigure}
\begin{subfigure}[Approximation to $u$]{
		\includegraphics[scale=0.5, angle=0]{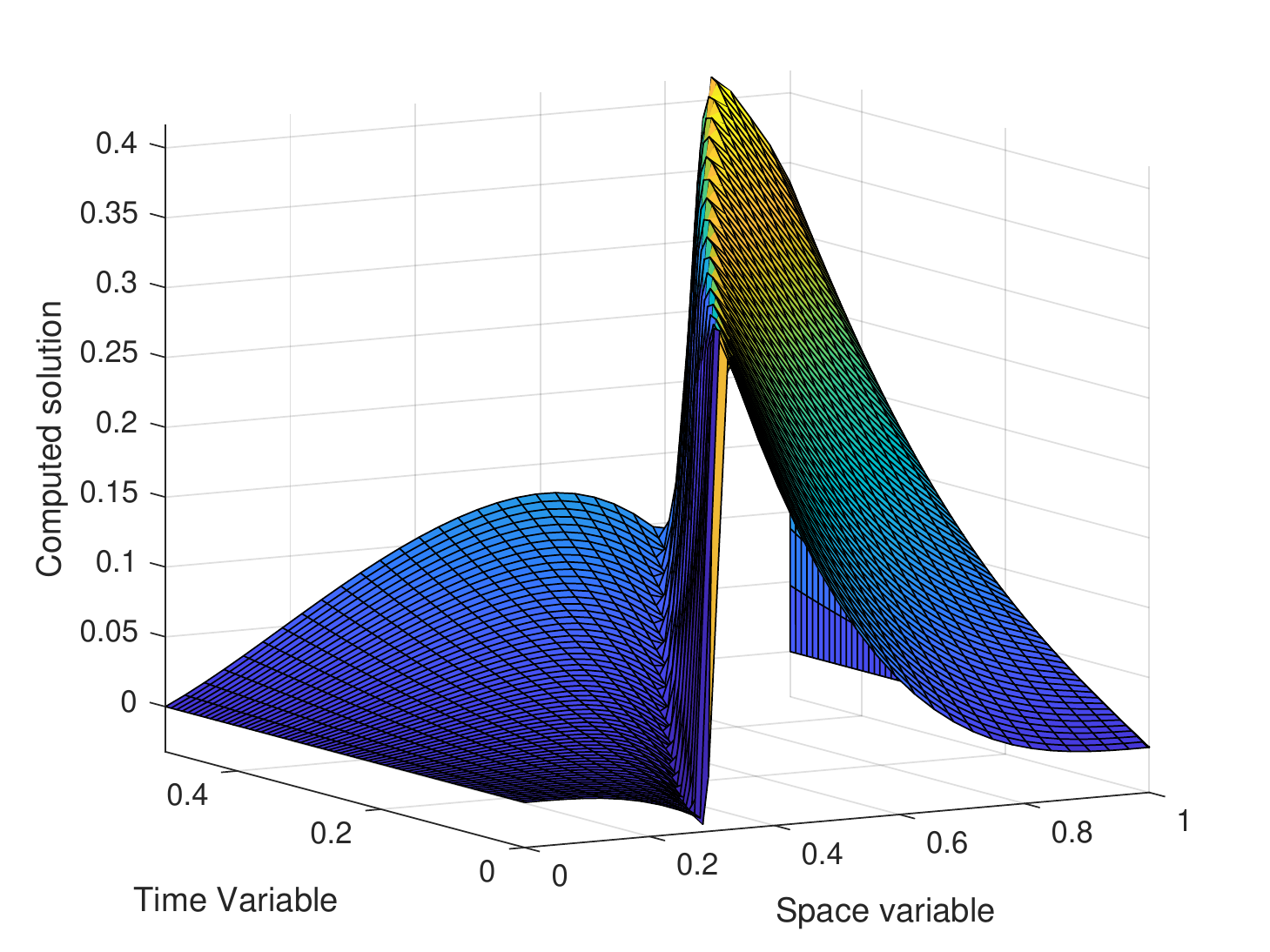}
		}
	\end{subfigure}
}
	\caption{Example~\ref{ex2}: Numerical approximations to $y$ and $u$ with $\vr=2^{-12}$ and $N=M=64$}
	\label{fig:ex2-solutions}
 \end{figure}

\begin{table}[h]
\caption{Example~\ref{ex2}: Maximum two-mesh global differences and orders of convergence for the function $y$}
\begin{center}{\tiny \label{tb:ex2}
\begin{tabular}{|c||c|c|c|c|c|c|c|}
 \hline  & N=M=32 & N=M=64 & N=M=128 & N=M=256 & N=M=512 & N=M=1024 & N=M=2048 \\
\hline \hline $\vr=2^{0}$
&1.071E-02 &1.000E-02 &6.887E-03 &5.201E-03 &3.621E-03 &2.685E-03 &1.883E-03 \\
&0.098&0.539&0.405&0.522&0.431&0.512&
\\ \hline $\vr=2^{-2}$
&3.590E-03 &4.569E-03 &3.007E-03 &2.481E-03 &1.694E-03 &1.314E-03 &9.131E-04 \\
&-0.348&0.603&0.278&0.550&0.367&0.525&
\\ \hline $\vr=2^{-4}$
&7.415E-03 &3.990E-03 &2.255E-03 &1.247E-03 &7.917E-04 &6.614E-04 &4.582E-04 \\
&0.894&0.823&0.854&0.656&0.259&0.530&
\\ \hline $\vr=2^{-6}$
&1.125E-02 &7.051E-03 &4.061E-03 &2.202E-03 &1.154E-03 &5.894E-04 &2.964E-04 \\
&0.674&0.796&0.883&0.933&0.969&0.992&
\\ \hline $\vr=2^{-8}$
&1.425E-02 &9.726E-03 &6.373E-03 &3.921E-03 &2.277E-03 &1.248E-03 &6.590E-04 \\
&0.551&0.610&0.701&0.784&0.867&0.922&
\\ \hline $\vr=2^{-10}$
&1.519E-02 &1.088E-02 &7.564E-03 &5.138E-03 &3.344E-03 &2.063E-03 &1.198E-03 \\
&0.482&0.524&0.558&0.620&0.697&0.784&
\\ \hline $\vr=2^{-12}$
&1.543E-02 &1.121E-02 &7.945E-03 &5.636E-03 &3.911E-03 &2.651E-03 &1.725E-03 \\
&0.461&0.497&0.495&0.527&0.561&0.620&
\\ \hline $\vr=2^{-14}$
&1.550E-02 &1.130E-02 &8.048E-03 &5.777E-03 &4.102E-03 &2.893E-03 &2.007E-03 \\
&0.456&0.490&0.478&0.494&0.504&0.528&
\\
\hline &\vdots &\vdots &\vdots &\vdots &\vdots &\vdots &\vdots \\
&&&&&&&
\\ \hline $\vr=2^{-26}$
 &{\bf 1.552E-02} &{\bf 1.133E-02} &{\bf 8.083E-03 }&{\bf 5.827E-03} &{\bf 4.172E-03} &{\bf 2.989E-03 }&{\bf 2.135E-03} \\
&0.454&0.487&0.472&0.482&0.481&0.485&
\\ \hline $D^{N,M}$
&1.552E-02 &1.133E-02 &8.083E-03 &5.827E-03 &4.172E-03 &2.989E-03 &2.135E-03 \\
$P^{N,M}$ &0.454&0.487&0.472&0.482&0.481&0.485&\\ \hline \hline
\end{tabular}}
\end{center}
\end{table}

 \begin{figure}[h!]
\centering
\resizebox{\linewidth}{!}{
	\begin{subfigure}[Approximation to $y_1$]{
		\includegraphics[scale=0.5, angle=0]{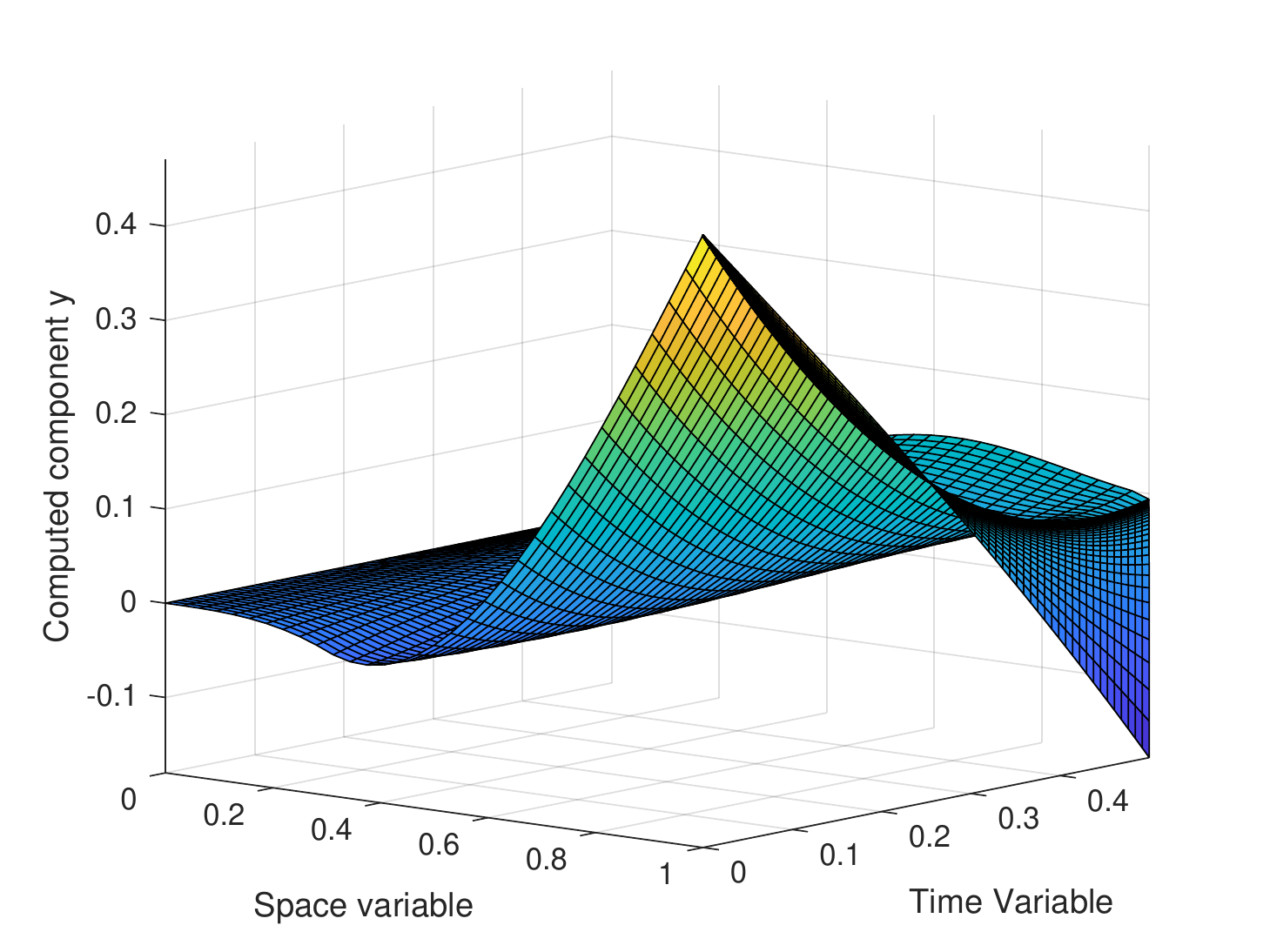}
		}
    \end{subfigure}
\begin{subfigure}[Improved approximation to $u$]{
		\includegraphics[scale=0.5, angle=0]{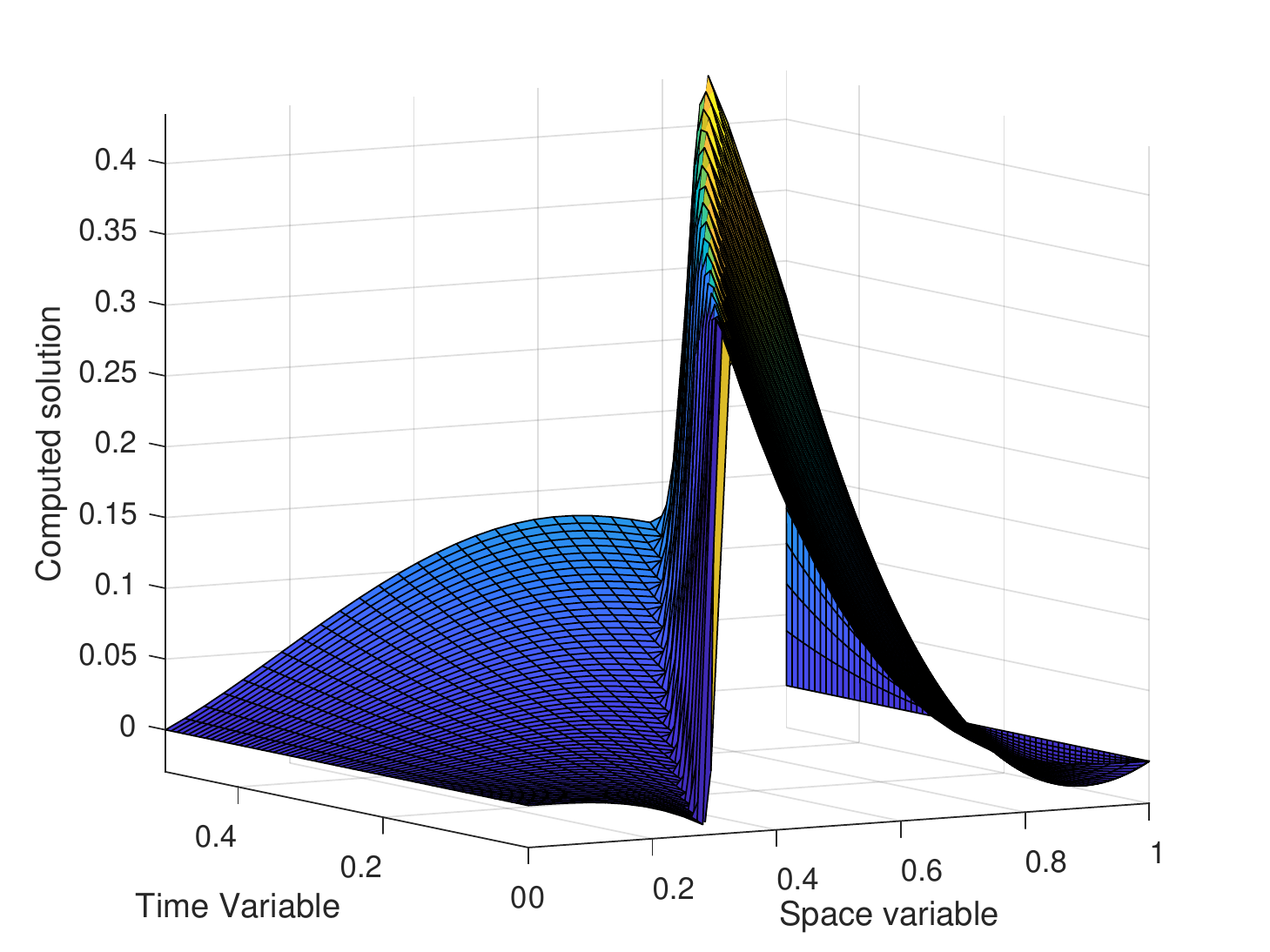}
		}
	\end{subfigure}
}
	\caption{ Example~\ref{ex2}: Numerical approximation to $y_1$ and improved approximation to $u$ for $\vr=2^{-12}$ and $N=M=64$}
	\label{fig:ex2-solution+2md-y1}
 \end{figure}

\begin{table}[h]
\caption{Example~\ref{ex2}: Uniform two-mesh global differences and orders of convergence for the function $y_1$}
\begin{center}{\tiny \label{tb:ex2-y1}
\begin{tabular}{|c||c|c|c|c|c|c|c|}
 \hline  & N=M=32 & N=M=64 & N=M=128 & N=M=256 & N=M=512 & N=M=1024 & N=M=2048
\\ \hline $D^{N,M}$
&1.403E-02 &8.451E-03 &4.856E-03 &2.669E-03 &1.413E-03 &7.361E-04 &3.789E-04 \\
$P^{N,M}$ &0.731&0.799&0.863&0.917&0.941&0.958&\\ \hline \hline
\end{tabular}}
\end{center}
\end{table}

\begin{example} \label{ex3}
 Consider the following test problem
\[
\begin{array}{l}
-\ve  u_{xx} + (1+t) u_x + u_t= 4x(1-x)t+t^2, \quad (x,t) \in (0,1)\times (0,2], \\
 u(x,0)=-2, 0 \leq x < 0.3, \quad    u(x,0)=1, \ 0.3 \leq x \leq 1,         \\
 u(0,t)=-2, \quad    u(1,t) =1, \ 0 < t \leq 2.
\end{array}
\]
\end{example}

The initial condition is discontinuous at $d=0.3$, $[\phi ' ](0.3) = 0$ and the characteristic curve $\Gamma^*$ is now given by $d(t)=t+t^2/2+0.3$. The final time has been chosen large enough so that the interior layer interacts with the boundary layer. Then, we use a piecewise uniform mesh in time (\ref{time-Sh-mesh}) by computing $T^*$  with $d(T^*)=1$. In this example, it is given by
\[
T^*=\sqrt{1+2(1-d)}-1 \approx 0.5492.
\]
In Figure~\ref{fig:ex3-solutions}  a prominent layer near the boundary $x=1$  is observed for $t\ge T^*.$ Error bounds when the layers interact are discussed in \S~\ref{sec:LayerInteract} and it is proved that Theorem~\ref{th_a(t)} also applies.
The maximum two-mesh global differences and the orders of convergence are given in Table~\ref{tb:ex3} and it is observed that it is a uniformly and globally convergent scheme having almost first-order. All these numerical results agree with  the error bound established in Theorem~\ref{th_a(t)}.

 \begin{figure}[h!]
\centering
\resizebox{\linewidth}{!}{
	\begin{subfigure}[Approximation to $y$]{
		\includegraphics[scale=0.5, angle=0]{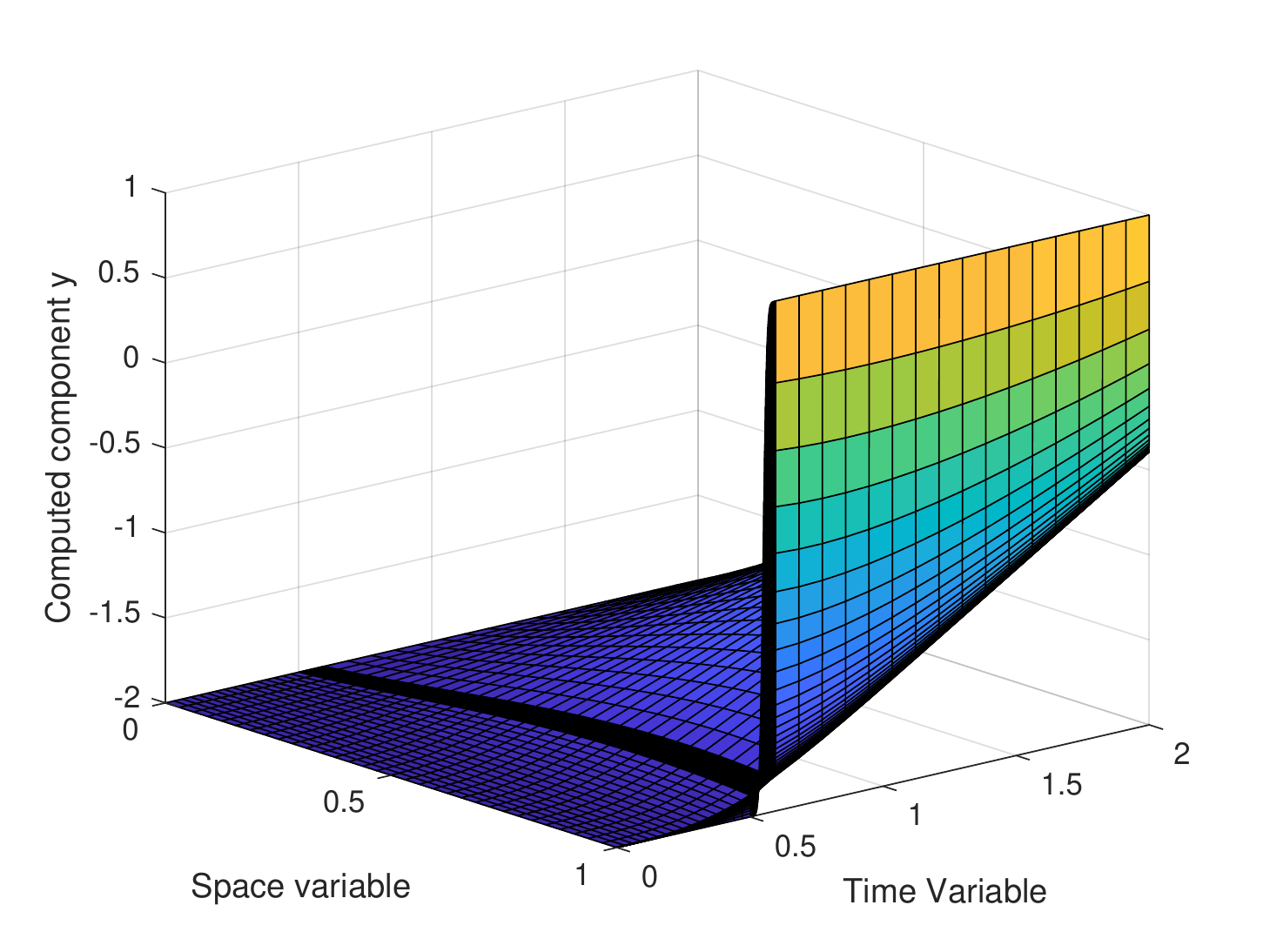}
		}
    \end{subfigure}
\begin{subfigure}[Approximation to $u$]{
		\includegraphics[scale=0.5, angle=0]{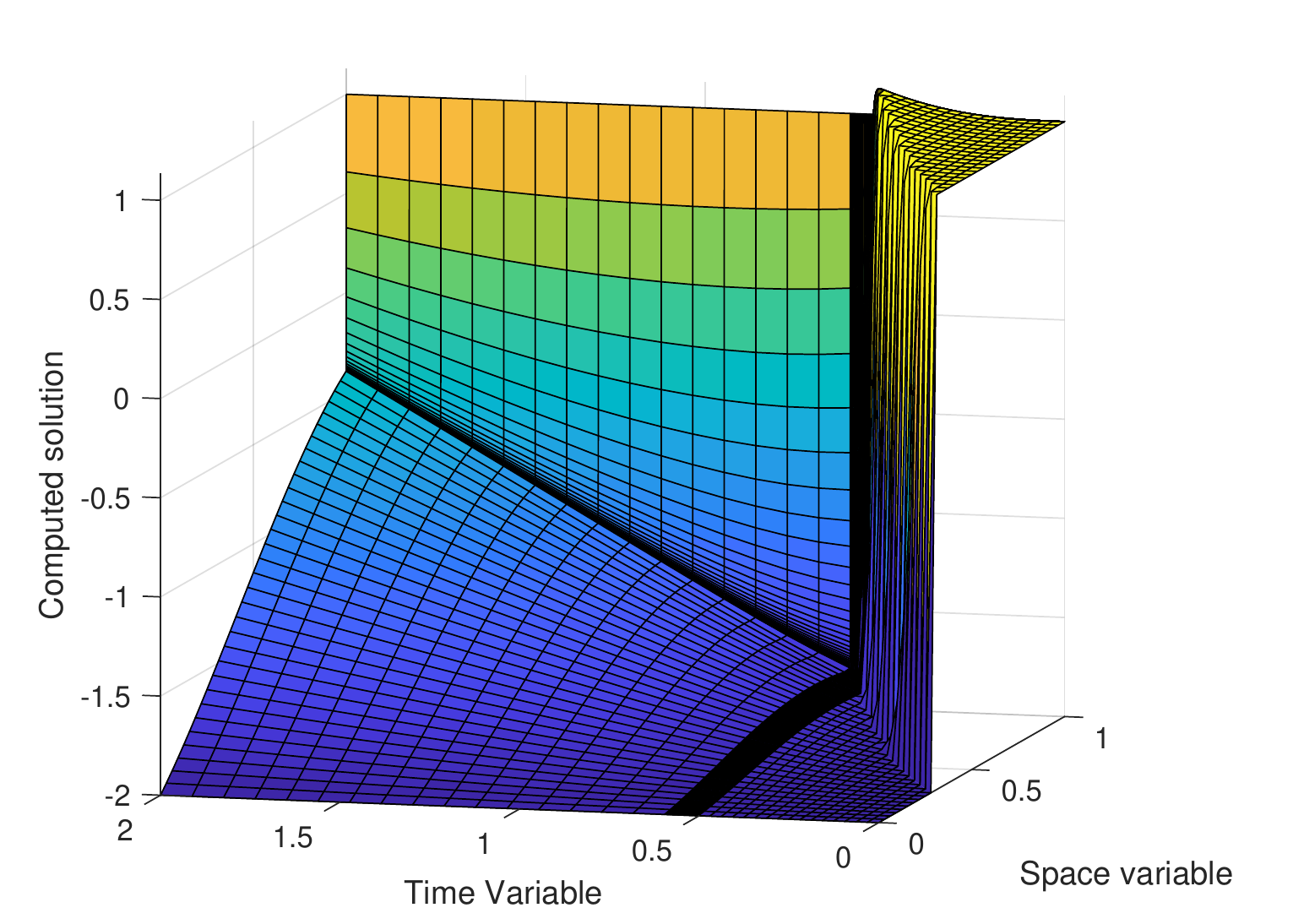}
		}
	\end{subfigure}
}
	\caption{Example~\ref{ex3}: Numerical approximations to $y$ and $u$ with $\vr=2^{-12}$ and $N=M=64$}
	\label{fig:ex3-solutions}
 \end{figure}

\begin{table}[h]
\caption{Example~\ref{ex3}: Maximum two-mesh global differences and orders of convergence for the function $y$}
\begin{center}{\tiny \label{tb:ex3}
\begin{tabular}{|c||c|c|c|c|c|c|c|}
 \hline  & N=M=32 & N=M=64 & N=M=128 & N=M=256 & N=M=512 & N=M=1024 & N=M=2048 \\
\hline \hline $\vr=2^{0}$
&2.899E-02 &3.735E-02 &1.233E-02 &4.263E-03 &1.804E-03 &8.821E-04 &4.364E-04 \\
&-0.365&1.599&1.532&1.241&1.032&1.015&
\\ \hline $\vr=2^{-2}$
&4.847E-02 &2.486E-02 &1.259E-02 &6.342E-03 &3.182E-03 &1.594E-03 &7.977E-04 \\
&0.963&0.981&0.990&0.995&0.997&0.999&
\\ \hline $\vr=2^{-3}$
&7.498E-02 &{\bf 5.263E-02 }&2.677E-02 &1.350E-02 &6.782E-03 &3.399E-03 &1.702E-03 \\
&0.511&0.975&0.988&0.993&0.996&0.998&
\\ \hline $\vr=2^{-4}$
&7.313E-02 &4.724E-02 &2.811E-02 &{\bf 1.670E-02} &{\bf 9.579E-03} &{\bf 5.411E-03 }&{\bf 3.145E-03} \\
&0.630&0.749&0.751&0.802&0.824&0.783&
\\ \hline $\vr=2^{-6}$
&7.543E-02 &4.653E-02 &2.755E-02 &1.637E-02 &9.407E-03 &5.330E-03 &2.973E-03 \\
&0.697&0.756&0.751&0.799&0.820&0.842&
\\ \hline $\vr=2^{-8}$
&7.872E-02 &4.619E-02 &2.731E-02 &1.621E-02 &9.305E-03 &5.274E-03 &2.942E-03 \\
&0.769&0.758&0.753&0.801&0.819&0.842&
\\ \hline $\vr=2^{-10}$
&7.993E-02 &4.718E-02 &2.746E-02 &1.617E-02 &9.282E-03 &5.260E-03 &2.934E-03 \\
&0.761&0.781&0.764&0.801&0.819&0.842&
\\ \hline $\vr=2^{-12}$
&7.996E-02 &4.778E-02 &2.797E-02 &1.616E-02 &9.278E-03 &5.256E-03 &2.932E-03 \\
&0.743&0.772&0.791&0.801&0.820&0.842&
\\ \hline $\vr=2^{-14}$
&{\bf 8.020E-02 }&4.796E-02 &2.821E-02 &1.619E-02 &9.278E-03 &5.256E-03 &2.932E-03 \\
&0.742&0.765&0.801&0.804&0.820&0.842&
\\
\hline &\vdots &\vdots &\vdots &\vdots &\vdots &\vdots &\vdots \\
&&&&&&&
\\ \hline $\vr=2^{-18}$
&8.008E-02 &4.806E-02 &{\bf 2.838E-02} &1.632E-02 &9.278E-03 &5.255E-03 &2.932E-03 \\
&0.736&0.760&0.798&0.815&0.820&0.842&
\\
\hline &\vdots &\vdots &\vdots &\vdots &\vdots &\vdots &\vdots \\
&&&&&&&
\\ \hline $\vr=2^{-26}$
&7.995E-02 &4.799E-02 &2.836E-02 &1.636E-02 &9.293E-03 &5.254E-03 &2.931E-03 \\
&0.736&0.759&0.794&0.816&0.823&0.842&
\\ \hline $D^{N,M}$
&8.020E-02 &5.263E-02 &2.838E-02 &1.670E-02 &9.579E-03 &5.411E-03 &3.145E-03 \\
$P^{N,M}$ &0.608&0.891&0.765&0.802&0.824&0.783&\\ \hline \hline
\end{tabular}}
\end{center}
\end{table}

\subsection{Extensions} \label{sec:numer-extension}

\begin{example} \label{ex4}
Consider the following text problem
\[
\begin{array}{l}
-\ve  u_{xx} + (1+t^2) u_x + u_t= 4x(1-x)t+t^2, \quad (x,t) \in (0,1)\times (0,0.5], \\
 u(x,0)=-2x, 0 \leq x < d, \quad    u(x,0)=1-x^2, \ d \leq x \leq 1,         \\
 u(0,t)= 4t^2, \quad     u(1,t) = t(t+0.5), \ 0 < t \leq 0.5; \\
d := \min \{0.3, \sqrt{\vr} \}.
\end{array}
\]
\end{example}
In this example $[\phi](d)\ne 0, \ [\phi ' ](d) \ne 0$, $d(t)=t+t^3/3+d$ and the interior and boundary  layers do not merge (see Figure~\ref{fig:ex4-solutions}.) The aim of this test problem is to show numerically that the analytical/numerical method proposed in this paper can also be applied when the distance of the discontinuity point $(d,0)$ of the initial condition to the point $(0,0)$ depends on the singular perturbation parameter   and $1> d \geq C\sqrt{\vr}$. Note that the method will fail if $d = \vr ^p, p> 0.5$.

The numerical results in Table~\ref{tb:ex4} suggest that the  numerical approximations converge globally and uniformly with order 0.5. For other test examples with $d=O( \sqrt{\vr}) $  but $[\phi ' ](d) = 0$, we have observed numerically that the numerical approximations converge with almost first order. The proof of these observed orders of  convergence  for this problem class (with $d=O( \sqrt{\vr}) $) is an open question.

 \begin{figure}[h!]
\centering
\resizebox{\linewidth}{!}{
	\begin{subfigure}[Approximation to $y$]{
		\includegraphics[scale=0.5, angle=0]{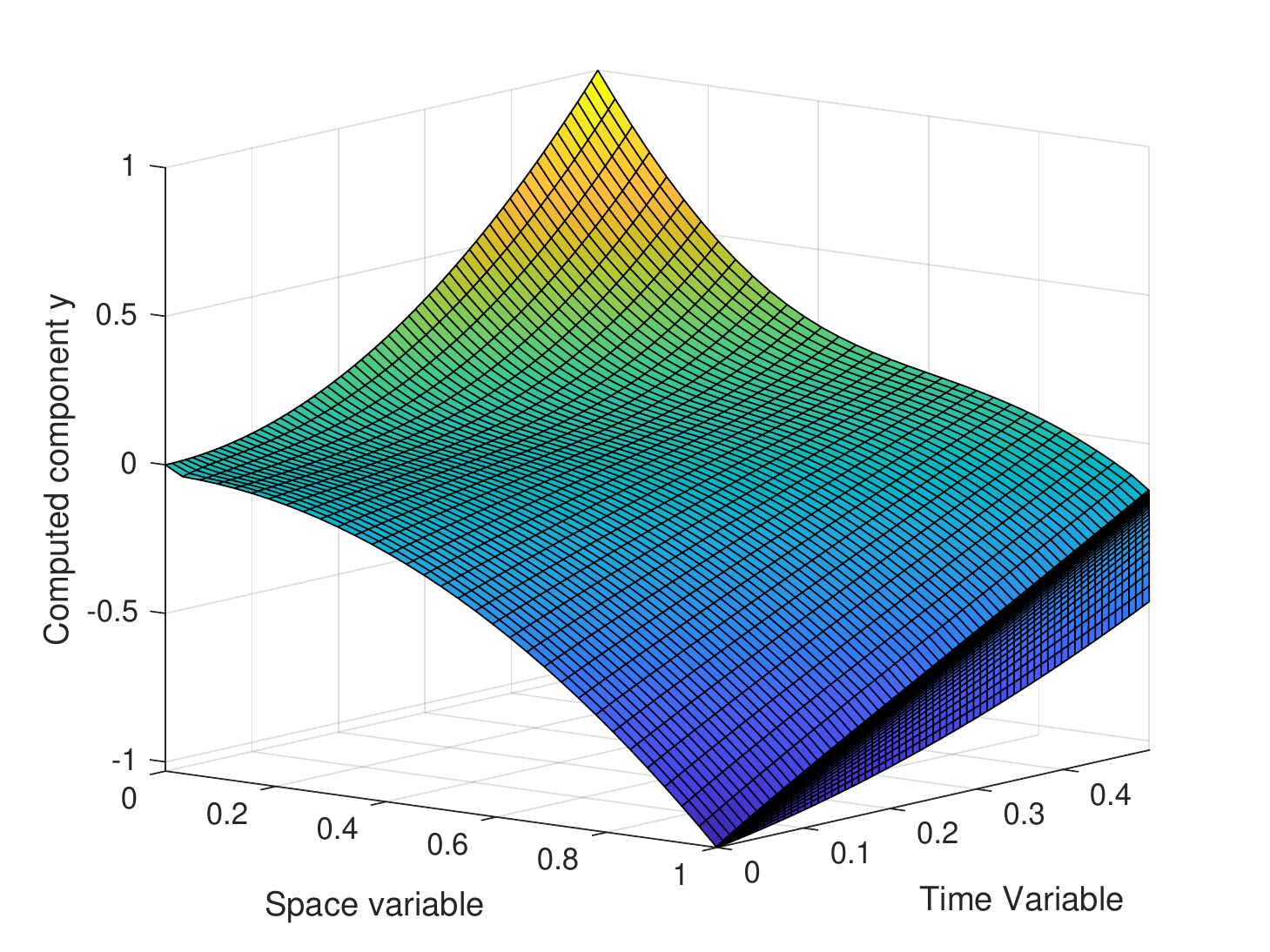}
		}
    \end{subfigure}
\begin{subfigure}[Approximation to $u$]{
		\includegraphics[scale=0.5, angle=0]{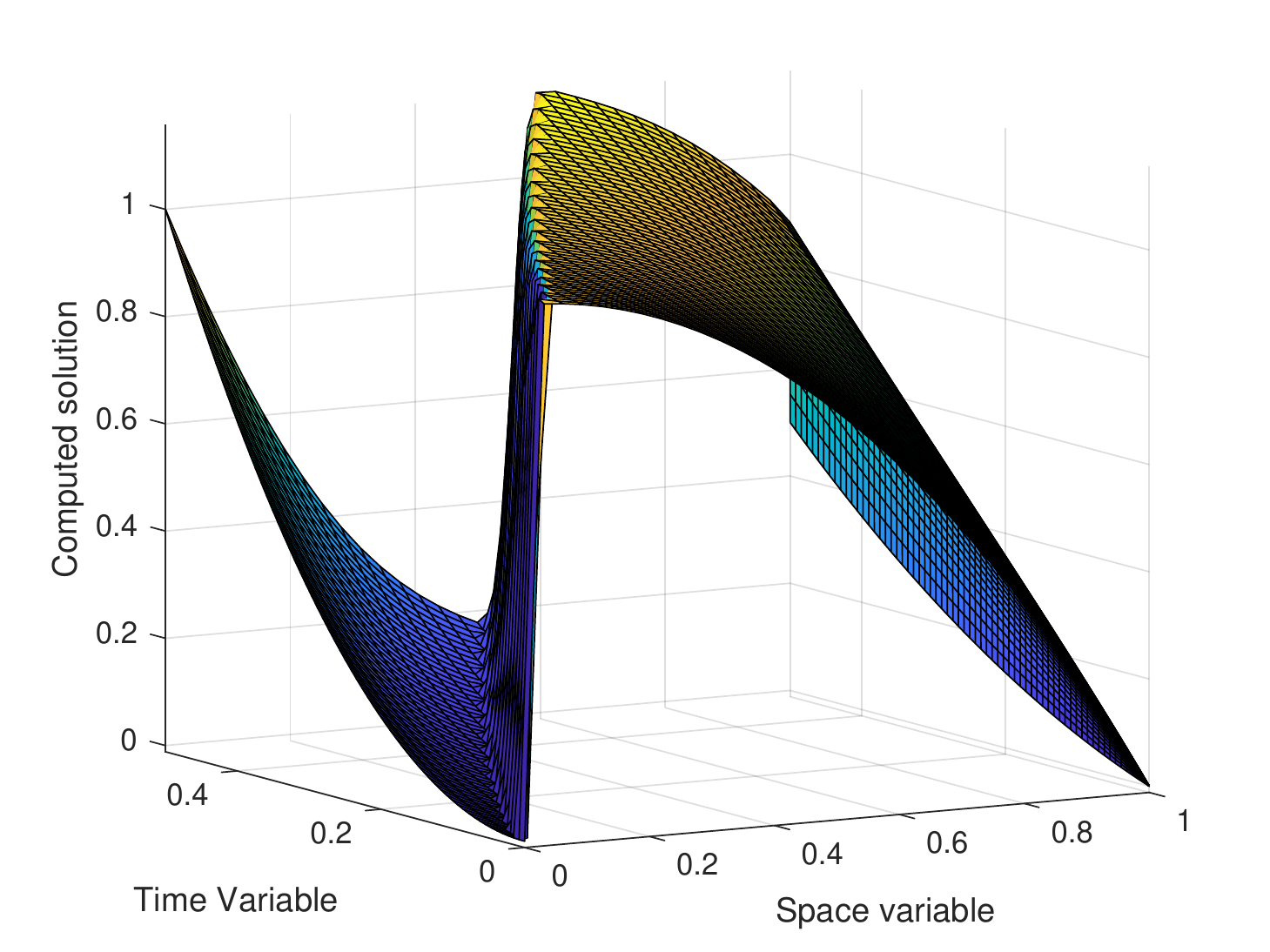}
		}
	\end{subfigure}
}
	\caption{ Example~\ref{ex4}: Numerical approximations to $y$ and $u$ with $\vr=2^{-12}$ and $N=M=64$}
	\label{fig:ex4-solutions}
 \end{figure}

\begin{table}[h]
\caption{Example~\ref{ex4}: Maximum two-mesh global differences and orders of convergence for the function $y$}
\begin{center}{\tiny \label{tb:ex4}
\begin{tabular}{|c||c|c|c|c|c|c|c|}
 \hline  & N=M=32 & N=M=64 & N=M=128 & N=M=256 & N=M=512 & N=M=1024 & N=M=2048 \\
\hline \hline $\vr=2^{0}$
&1.759E-02 &1.215E-02 &8.027E-03 &{\bf 6.110E-03} & {\bf 4.246E-03} & {\bf 3.134E-03} &{\bf 2.197E-03} \\
&0.534&0.598&0.394&0.525&0.438&0.512&
\\ \hline $\vr=2^{-2}$
&7.582E-03 &5.363E-03 &3.499E-03 &2.916E-03 &1.988E-03 &1.534E-03 &1.065E-03 \\
&0.500&0.616&0.263&0.553&0.374&0.526&
\\ \hline $\vr=2^{-4}$
&1.518E-02 &7.788E-03 &3.925E-03 &2.172E-03 &1.249E-03 &7.257E-04 &5.858E-04 \\
&0.963&0.988&0.854&0.798&0.784&0.309&
\\ \hline $\vr=2^{-6}$
&2.646E-02 &1.492E-02 &8.494E-03 &4.779E-03 &2.652E-03 &1.455E-03 &7.905E-04 \\
&0.826&0.813&0.830&0.849&0.866&0.880&
\\ \hline $\vr=2^{-8}$
&3.636E-02 &1.969E-02 &1.076E-02 &5.774E-03 &3.540E-03 &2.057E-03 &1.122E-03 \\
&0.885&0.872&0.898&0.706&0.783&0.875&
\\ \hline $\vr=2^{-10}$
&4.039E-02 &2.084E-02 &1.121E-02 &5.998E-03 &3.755E-03 &2.668E-03 &1.717E-03 \\
&0.954&0.894&0.903&0.676&0.493&0.635&
\\ \hline $\vr=2^{-12}$
&4.181E-02 &2.112E-02 &1.132E-02 &6.059E-03 &3.255E-03 &2.252E-03 &1.783E-03 \\
&0.985&0.900&0.901&0.896&0.532&0.337&
\\ \hline $\vr=2^{-14}$
&4.236E-02 &2.121E-02 &1.134E-02 &6.073E-03 &3.265E-03 &1.753E-03 &1.250E-03 \\
&0.998&0.903&0.901&0.895&0.897&0.488&
\\
\hline &\vdots &\vdots &\vdots &\vdots &\vdots &\vdots &\vdots \\
&&&&&&&
\\ \hline $\vr=2^{-26}$
&{\bf 4.281E-02} & {\bf 2.127E-02} & {\bf 1.135E-02} &6.078E-03 &3.268E-03 &1.755E-03 &9.405E-04 \\
&1.009&0.906&0.901&0.895&0.897&0.900&
\\ \hline $D^{N,M}$
&4.281E-02 &2.127E-02 &1.135E-02 &6.110E-03 &4.246E-03 &3.134E-03 &2.197E-03 \\
$P^{N,M}$ &1.009&0.906&0.894&0.525&0.438&0.512&\\ \hline \hline
\end{tabular}}
\end{center}
\end{table}

\begin{example} \label{ex5}
Finally, we consider the following  test problem
\[
\begin{array}{l}
-\ve  u_{xx} + (1+x^2) u_x + u_t= 4x(1-x)t+t^2, \quad (x,t) \in (0,1)\times (0,0.5], \\
 u(x,0)=-2, 0 \leq x < 0.1, \quad    u(x,0)=1, \ 0.1 \leq x \leq 1,         \\
 u(0,t)=-2, \quad    u(1,t) =1, \ 0 < t \leq 0.5.
\end{array}
\]
\end{example}
In this example $[\phi](0.1)=3\ne 0, \ [\phi ' ](d)=0$ and the interior and boundary  layers do not merge. Note that the convective coefficient depends on the space variable, while our error analysis assumes that $a=a(t).$  The characteristic curve $\Gamma^*$ is the solution of the initial value problem $d'(t)=(1+d^2(t))$, with $d(0)=0.1$, which is given by $d(t)=(0.1+\tan t)/(1-0.1 \tan t).$

 In Table~\ref{tb:ex5} we observe that the  numerical method is not a parameter-uniform method when $a$ depends  on the space variable. In order to design a uniformly convergent method in the case where $ a=a(x,t)$, a more sophisticated scheme is required \cite{arxiv}.

\begin{table}[h]
\caption{Example~\ref{ex5}: Maximum two-mesh global differences and orders of convergence for the function $y$}
\begin{center}{\tiny \label{tb:ex5}
\begin{tabular}{|c||c|c|c|c|c|c|c|}
 \hline  & N=M=32 & N=M=64 & N=M=128 & N=M=256 & N=M=512 & N=M=1024 & N=M=2048 \\
\hline \hline $\vr=2^{0}$
&1.978E-01 &6.835E-02 &3.367E-02 &4.361E-02 &1.478E-02 &4.979E-03 &2.029E-03 \\
&1.533&1.021&-0.373&1.561&1.570&1.295&
\\ \hline $\vr=2^{-2}$
&2.811E-02 &3.521E-02 &1.235E-02 &4.127E-03 &1.802E-03 &8.806E-04 &4.355E-04 \\
&-0.325&1.511&1.581&1.196&1.033&1.016&
\\ \hline $\vr=2^{-4}$
&1.387E-02 &6.663E-03 &3.219E-03 &1.529E-03 &7.176E-04 &3.341E-04 &1.532E-04 \\
&1.058&1.050&1.074&1.092&1.103&1.125&
\\ \hline $\vr=2^{-10}$
&1.580E-01 &9.484E-02 &4.791E-02 &2.418E-02 &1.215E-02 &6.088E-03 &3.046E-03 \\
&0.737&0.985&0.987&0.993&0.996&0.999&
\\ \hline $\vr=2^{-12}$
&6.620E-01 &2.759E-01 &1.172E-01 &5.656E-02 &2.859E-02 &1.431E-02 &7.158E-03 \\
&1.263&1.236&1.051&0.984&0.999&0.999&
\\ \hline $\vr=2^{-14}$
&1.391E+00 &6.782E-01 &2.748E-01 &1.113E-01 &6.205E-02 &3.102E-02 &1.553E-02 \\
&1.036&1.303&1.304&0.843&1.000&0.998&
\\ \hline $\vr=2^{-15}$
&{\bf 1.409E+00 }&{\bf 7.023E-01 }&3.482E-01 &1.451E-01 &8.928E-02 &4.509E-02 &2.255E-02 \\
&1.004&1.012&1.263&0.701&0.985&1.000&
\\ \hline $\vr=2^{-19}$
&4.540E-03 &2.705E-03 &{\bf 1.560E+00 }&4.250E-01 &1.997E-01 &2.022E-01 &9.535E-02 \\
&0.747&-9.171&1.876&1.089&-0.018&1.085&
\\ \hline $\vr=2^{-20}$
&4.540E-03 &2.705E-03 &1.152E+00 &{\bf 4.582E-01} &6.427E-01 &3.221E-01 &1.370E-01 \\
&0.747&-8.734&1.330&-0.488&0.996&1.234&
\\ \hline $\vr=2^{-23}$
&4.540E-03 &2.705E-03 &1.576E-03 &9.104E-04 &{\bf 1.602E+00} &{\bf 8.004E-01} &3.727E-01 \\
&0.747&0.779&0.792&-10.781&1.001&1.103&
\\ \hline $\vr=2^{-26}$
&4.540E-03 &2.705E-03 &1.576E-03 &9.105E-04 &5.055E-02 &2.521E-02 &{\bf 1.559E+00} \\
&0.747&0.779&0.792&-5.795&1.003&-5.950&
\\ \hline $D^{N,M}$
&1.409E+00 &7.023E-01 &1.560E+00 &4.582E-01 &1.602E+00 &8.004E-01 &1.559E+00 \\
$P^{N,M}$ &1.004&-1.151&1.767&-1.806&1.001&-0.961&\\ \hline \hline
\end{tabular}}
\end{center}
\end{table}

\section{Appendix: A set of singular functions}

Below, we define a set of functions $\{  \psi _i \} _{i=0}^4$ such that $ L \psi _i=0$; $  \psi _i \in  C^{i-1+\gamma } (\bar {  Q}), \ i \geq 1$ and each function $ \psi _i$ is smooth within the open region $  Q \setminus \Gamma ^*$.

Define the two singular functions  (see  \cite[(10)]{bobisud})
\begin{equation}\label{zero}
 \psi _0(x,t) := \erfc \left ( \frac{d(t)-x}{2\sqrt{\ve t}} \right) ,\quad  E(x,t) := e^{-\frac{(x-d(t))^2}{4\ve t}}.
\end{equation}
Then we explicitly write out the derivatives of these two functions
\begin{eqnarray*}
\frac{\partial  \psi _0}{\partial x} &=&  \frac{1}{\sqrt{\ve \pi t}}  E  , \quad  \frac{\partial  E}{\partial x} =  \frac{d(t)-x}{2\ve t}  E,\\
\ve \frac{\partial ^2 \psi _0}{\partial x ^2} &=&  \frac{d(t)-x}{2t\sqrt{\ve \pi t}}  E, \quad \ve \frac{\partial ^2  E}{\partial x ^2} = \Bigl( \frac{(d(t)-x)^2}{2\ve  t} -1 \Bigr)\frac{ E}{2t},\\\\
\frac{\partial  \psi _0}{\partial t} &=&  \frac{1}{\sqrt{\ve \pi t}}  \Bigl(\frac{(d(t)-x)-2ta}{2 t} \Bigr) E,
\\
\frac{\partial  E}{\partial t} &=&  \frac{(d(t)-x)}{2\ve t}  \Bigl(\frac{(d(t)-x)}{2t} -a(t)\Bigr) E = \frac{\sqrt{ \pi }(d(t)-x)}{2\sqrt{\ve  t}}\frac{\partial  \psi _0}{\partial t}.
\end{eqnarray*}
Hence,
$
 L \psi _0=0.
$
We define the continuous function
\begin{equation}\label{one}
 \psi _1(x,t) := (d(t)-x) \psi _0 -2\frac{\sqrt{\ve t}}{\sqrt{\pi}}  E,
\end{equation}
with
\[
 \frac{\partial  \psi _1}{\partial t} = a(t) \psi _0 -\frac{\sqrt{\ve }}{\sqrt{\pi t}}  E; \quad
\frac{\partial  \psi _1}{\partial x} = -  \psi _0 \quad \hbox{so that}\quad  L \psi _1\equiv 0.
\]
We now  define the remaining functions:
\begin{equation}\label{singular-functions}
 \psi _i := (d(t)-x)  \psi _{i-1}+2\ve t (i-1)  \psi _{i-2},\qquad  i=2,3,4;
\end{equation}
and we can check that  for $ i=1,2,3,4$
\begin{align*}
\frac{\partial  \psi _i}{\partial x} = -i  \psi _{i-1},\ L\psi _i =0 \quad \hbox{and} \quad \psi _i \in C^{i-1+\gamma} (\bar {  Q)}.
\end{align*}
Either side of $x=d$, we have the Taylor expansions for the initial condition
\[
\phi (x) =   \sum _{i=0}^4 \phi ^{(i)} (d^*)   \frac{(x-d)^i}{i!} +   R_0(x),\quad d^*= \Bigl \{ \begin{array}{ll} d^- \quad \hbox{ if } x <  d \\ d^+ \quad \hbox{ if } x>d \end{array} ;
\]
 with $  R_0(x) \in C^4(0,1)$. Hence,
we present the following expansion
\begin{equation}\label{expansion}
 u (x,t)= 0.5 \sum _{i=0}^4 [\phi ^{(i)} ] (d) \frac{(-1)^i}{i!}  \psi _i  (x,t) +  R (x,t).
\end{equation}
Note that  for $ i=1,2,3,4$
\[
\frac{\partial ^i  \psi _i}{\partial x^i} = (-1)^ii!  \psi _{0} \quad \hbox{and} \quad [ \psi _{0} ] (d)=2,
\]
which implies that $[ u ^{(i)} ] (d,0)=[\phi ^{(i)} ] (d)$.

Define the paramaterized exponential function
\[
 E_\gamma (x,t) := e^{-\frac{\gamma (x-d(t))^2}{4\ve t}},\qquad  0 < \gamma < 1.
\]
 Using the inequality $\erfc(z) \leq C e^{-z^2} \leq C e^{\gamma ^2/4}e^{-\gamma z},  \forall z \ge 0 $
it follows that
\begin{subequations}\label{bounds-singular-functions}
\begin{align}
& \vert  \psi _0(x,t) \vert  \leq C, \\
&  \Bigl \vert \frac{\partial ^j }{\partial t ^j }  \psi _0(x,t)\Bigr \vert, \Bigl \vert \frac{\partial ^j }{\partial t ^j }  E(x,t)\Bigr \vert   \leq  C
 \left( \frac{1}{t} +\frac{1}{\sqrt {\ve t}}\right)^j E_\gamma (x,t);\quad j=1,2; \label{bound-A}\\
& \Bigl \vert \frac{\partial ^i }{\partial x ^i }  \psi _0(x,t)\Bigr \vert, \Bigl \vert \frac{\partial ^i }{\partial x^i }  E(x,t)\Bigr \vert     \leq C\left(\frac{1}{\sqrt {\ve t}}\right) ^{i} E_\gamma (x,t), \ 1\leq i \leq 4;\label{bound-B}\\
& \Bigl \vert \frac{\partial  }{\partial t  }  \psi _1(x,t)\Bigr \vert \leq   C  \left(1 +\sqrt {\frac{\ve}{ t}}\right)  E (x,t)+C, \label{bound-C}\\
& \Bigl\vert \frac{\partial ^2 }{\partial t ^2 }  \psi _1(x,t)\Bigr \vert    \leq   C\left (  \frac1{t}+\frac{1}{t}\sqrt{\frac{\ve}{t}}+\frac1{\sqrt{\vr t}} \right) E_\gamma (x,t)+C, \label{bound-D}
 \\
& \vert \psi _1(x,t) \vert
 \leq C, \, \,   \Bigl \vert \frac{\partial ^i }{\partial x ^i } \psi _1(x,t) \Bigr \vert    \leq  C\left(\frac{1}{\sqrt {\ve t}}\right)^{i-1} E_\gamma (x,t), \, 1 \leq i \leq 4. \label{bound-E}
\end{align}
\end{subequations}

For the next terms, we can also establish the bounds
\begin{subequations}\label{more-bounds-singular-functions}
\begin{equation}
\Bigl \vert \frac{\partial  }{\partial t  }  \psi _2(x,t)\Bigr \vert \leq C,\quad \Bigl \vert
\frac{\partial ^2 }{\partial x ^2 }   \psi _2(x,t)\Bigr \vert \leq  C;
\end{equation}
 on the second time derivatives
\begin{align}
\Bigl \vert \frac{\partial ^2 }{\partial t ^2 }  \psi _2(x,t)\Bigr \vert  &\leq  C
\ve \left( \frac{1}{t} +\frac{1}{\sqrt {\ve t}}+ \frac{1}{\ve} \right) E_\gamma (x,t) +C \nonumber \\
& \leq
C \left (1+ \frac{\ve}{t} \right) E_\gamma (x,t) +C, \\
\Bigl \vert \frac{\partial ^2 }{\partial t ^2 }  \psi _3(x,t)\Bigr \vert   &\leq
C \sqrt{\ve t} \left(1+ \frac{\ve}{t}\right) E_\gamma (x,t) +C
, \\
\Bigl \vert \frac{\partial ^2 }{\partial t ^2 }  \psi _4(x,t)\Bigr \vert  &\leq  C
\ve  E_\gamma (x,t)+C
\end{align}
 on the fourth space derivatives
\begin{equation}
 \Bigl \vert \frac{\partial ^4 }{\partial x^4 }  \psi _j(x,t)  \Bigr \vert \leq C(\sqrt{\ve t})^{j-4}  E_\gamma (x,t)+C, \quad j=2,3,4
\end{equation}
and on the third space derivatives
\begin{equation}
 \Bigl \vert \frac{\partial ^3 }{\partial x^3 }  \psi _2(x,t)  \Bigr \vert \leq  C
 \frac{1}{\sqrt {\ve t}}  E_\gamma (x,t)+C, \quad
\Bigl \vert \frac{\partial ^3 }{\partial x ^3 }  \psi _3(x,t)\Bigr \vert  , \Bigl \vert \frac{\partial ^3 }{\partial x ^3 }  \psi _4(x,t)  \Bigr \vert \leq  C.
\end{equation}
\end{subequations}

\end{document}